\documentclass[11pt]{article}

\usepackage[margin=1.2in]{geometry} 

\usepackage{amsmath, amsthm, amssymb, amsfonts} 
\usepackage{mathtools} 
\usepackage{yhmath}    

\usepackage{yfonts} 

\usepackage{graphicx}
\usepackage{epsfig, epstopdf} 
\usepackage{tikz} 
\usetikzlibrary{fadings, patterns, shadows.blur, shapes, tikzmark, calc, decorations.pathreplacing} 

\usepackage{array, multirow, tabularx, booktabs} 
\usepackage{enumitem}
\usepackage{hyperref} 
\usepackage{caption}

\hypersetup{
    colorlinks=true,
    linkcolor=blue!80!black,
    citecolor=blue!50!black,
    urlcolor=blue
}
\usepackage{titling} 
\usepackage{url} 
\usepackage{dsfont} 
\usepackage{physics} 
\usepackage{cancel} 
\usepackage{color} 
\usepackage{textcomp}
\usepackage{gensymb} 
\usepackage{extarrows} 
\usepackage{wasysym} 
\usepackage{multicol} 
\usepackage{ytableau} 
\usepackage{float} 
\usepackage[utf8]{inputenc} 
\AtBeginDocument{\RenewCommandCopy\qty\SI}
\usepackage{biblatex} 

\theoremstyle{plain}  

\newtheorem{theorem}{Theorem}[section]
\newtheorem{corollary}[theorem]{Corollary}
\newtheorem{lemma}[theorem]{Lemma}
\newtheorem{proposition}[theorem]{Proposition}
\newtheorem{definition}[theorem]{Definition}
\newtheorem{remark}[theorem]{Remark}

\newtheorem{Con}[theorem]{Conjecture}

\newcommand{\eig}{\text{eig}} 

\title{Cutoff for q-deformed classical card shuffles in the type $A$ Iwahori--Hecke algebra}

\author{
Samira Arfaee \thanks{Stony Brook University, NY 11794. E-mail: \texttt{seyedehfatemeh.arfaeezarandi@stonybrook.edu}. Supported by the NSF grant DMS-2450510.} 
\and 
Bryan Wong \thanks{Stony Brook University, NY 11794. E-mail: \texttt{bryan.wong@stonybrook.edu}.}
}

\date{}

\begin{document}

\maketitle  

\begin{abstract}

We study the mixing behavior of three q-deformed card shuffles on $\mathcal{H}_q(S_n)$: the short systematic scan introduced in \cite{DiaconisRam2000}, the $q$--deformed random--to--random shuffle introduced in \cite{AxelrodFreedBraunerChiangComminsLang2024}, and the $q$--deformed $k$--star transposition shuffle, whose $q=1$ counterpart was studied in \cite{arfaee2025shuffling}. The first two chains were diagonalized in \cite{DiaconisRam2000} and \cite{AxelrodFreedBraunerChiangComminsLang2024}, respectively. Here, we diagonalize the $q$--deformed $k$--star transposition shuffle and use the spectra of the three chains to study their mixing behavior. For fixed $q>1$, we prove that all three exhibit total variation and $\ell^2$ cutoff. For the short systematic scan, these improve both the upper and lower bounds of Diaconis and Ram \cite{DiaconisRam2000}.


\end{abstract}

\section{Introduction}

In this paper, we study three random walks of the type $A$ Iwahori--Hecke algebra $\mathcal{H}_q(S_n)$. The first is the short systematic scan of Diaconis and Ram \cite{DiaconisRam2000}, obtained by applying adjacent Metropolis moves successively in one direction and then in reverse. The second is the $q$--random--to--random shuffle of Axelrod-Freed, Brauner, Chiang, Commins and Lang \cite{AxelrodFreedBraunerChiangComminsLang2024}, obtained by composing a $q$--random--to--bottom move with a $q$--bottom--to--random move. The third is the $q$--deformed $k$--star transposition shuffle, a Hecke-algebra analog of the $k$--star shuffle studied in \cite{arfaee2025shuffling}. All three chains have the Mallows measure as their stationary distribution and can be viewed as combinations of Metropolis moves encoded by elements of $\mathcal{H}_q(S_n)$. Although their transition operators are quite different, for fixed $q>1$ they turn out to have the same mixing scales.
The spectrum of the first two chains was computed in \cite{DiaconisRam2000} and \cite{AxelrodFreedBraunerChiangComminsLang2024}, respectively. Here we diagonalize the $q$--deformed $k$--star shuffle. For fixed $q>1$, we prove that all three chains exhibit the $\ell^2$--cutoff at $\frac12\bigl(n+\log_q n\bigr)$ with a constant-order window, and the total variation cutoff at $\frac{n}{2}$ with a window of at most order $\sqrt n$.

\subsection{History of problems}\label{sec:History}

The speed with which a deck of $n$ cards becomes random depends on the moves used to shuffle it. Viewing a shuffle as a random walk on the symmetric group $S_n$ gives a natural way to study this question, and many different shuffling models have been analyzed from this point of view; see, for example,
\cite{10.1214/20-AAP1632,
BD,
bernstein2019cutoff,
DiaconisSaloffCoste1993,
Diaconis1981,
Dieker2018,
Matthews1988,
MorrisNingPeres2014,
MorrisQin2017,
MosselPeresSinclair2004,
NamNestoridi2019,
SaloffCosteZuniga2008,
Subag,
Teyssier2019,
UyemuraReyes2002}.
We recall a few classical examples.


For random transpositions, two cards are chosen independently and uniformly at
random and then transposed. If the same card is chosen twice, the deck is left
unchanged. Diaconis and Shahshahani \cite{Diaconis1981} showed that this
shuffle exhibits the total variation cutoff at
$
\frac{1}{2}n\log n.
$ 

The random--to--random shuffle uses a different type of move. A card is chosen uniformly at random, removed from the deck, and inserted in a uniformly chosen position. Diaconis and Saloff-Coste \cite{DiaconisSaloffCoste1995} showed that its mixing time is of the order $n\log n$. Uyemura-Reyes \cite{UyemuraReyes2002} later obtained the bounds $ 
\frac12 n\log n \leq t_{\mathrm{mix}}\leq 4n\log n.
$ The upper bound was subsequently improved to $2n\log n$ by Saloff-Coste and Zuniga \cite{SaloffCosteZuniga2008}, and then to $1.5324n\log n$ by Morris and Qin \cite{MorrisQin2017}. Subag \cite{Subag} obtained a lower bound on the conjectured cutoff scale, while Dieker and Saliola \cite{Dieker2018} computed the complete spectrum of the transition matrix. Using this spectral description together with Subag's lower bound, Bernstein and Nestoridi \cite{bernstein2019cutoff} proved the total variation cutoff at $
\frac34 n\log n-\frac14 n\log\log n
$ with a window of order $n$, confirming the conjecture of Diaconis \cite{Diaconis2003}.


For star transpositions, one position is fixed, say position $n$. At each
step, a uniformly chosen card is transposed with the card in position $n$.
Diaconis \cite{PFlour} showed that the total variation cutoff for this shuffle
occurs at
$
n\log n.
$

Hence, the cutoff times of the three shuffles are ordered as
\[
\frac{1}{2}n\log n
\;<\;
\frac{3}{4}n\log n-\frac{1}{4}n\log\log n
\;<\;
n\log n.
\]

The random and star transposition shuffles are the two endpoints of the
$k$--star transposition shuffle introduced in
\cite{arfaee2025shuffling}. For $1\leq k\leq n$, let
$
A=\{n-k+1,\ldots,n\}.
$
Apart from the holding probability, the allowed moves are chosen uniformly
from
$
\bigcup_{j\in A}\{(i,j):1\leq i<j\}.
$
We proved that this shuffle exhibits the total variation cutoff at
$
\frac{2n-(k+1)}{2(n-1)}\,n\log n.
$
When $k=1$, this is the star transposition shuffle and the cutoff time is
$n\log n$, while $k=n$ gives random transpositions and the cutoff time is
$\frac{1}{2}n\log n$.

The shuffles above all have the uniform distribution as their stationary
distribution. There are also natural shuffling models with non-uniform
stationary measures. An example is the Mallows measure. For $w\in S_n$, let
$\ell(w)$ denote its number of inversions and set
\[
    \pi(w)=\frac{q^{\ell(w)}}{\sum_{v\in S_n}q^{\ell(v)}}.
\]
When $q>1$, permutations with more inversions are given higher weight.

A natural way to sample from this measure is to use Metropolis moves based on
adjacent transpositions. Write $\theta=q^{-1}$. Given an adjacent pair,
the transposition is always performed if it increases the number of inversions,
while a transposition that decreases the number of inversions is performed
with probability $\theta$. Diaconis and Ram \cite{DiaconisRam2000} observed
that these moves have an algebraic interpretation. More precisely, the
transition matrix for a Metropolis move corresponding to a simple
transposition is the matrix of left multiplication by a normalized generator
of the type $A$ Iwahori--Hecke algebra $\mathcal{H}_q(S_n)$; see
Theorem~\ref{thm:metropolis-hecke}. Thus, a sequence of Metropolis moves can be
viewed as multiplication by a product of elements of $\mathcal{H}_q(S_n)$.
This makes the representation theory of the Hecke algebra available for the
study of the chain. 

Diaconis and Ram used this connection to analyze the short systematic scan. 
For the chain started at the identity, they showed that
$
    \frac{n}{2}+\log_q n+c
$
scans are sufficient for convergence when $c$ is large, whereas before
$\frac{n}{4}$ scans the total variation distance tends to one.


Random--to--random also belongs to a broader algebraic family of shuffling operators. Reiner, Saliola and Welker \cite{ReinerSaliolaWelker2014} introduced the symmetrized shuffling operators, of which random--to--random is a special case. Lafreni{\`e}re \cite{Lafreniere2019} later studied the eigenvalues of this family. More recently, Axelrod-Freed, Brauner, Chiang, Commins and Lang \cite{AxelrodFreedBraunerChiangComminsLang2024} introduced a $q$--deformation of random--to--random in $\mathcal{H}_q(S_n)$ and computed its complete spectrum. They showed that the eigenvalues are polynomials in $q$ with non-negative integer coefficients, and setting $q=1$ recovers the spectrum of Dieker and Saliola \cite{Dieker2018}. Brauner, Commins, Grinberg and Saliola \cite{BCGS} later extended this construction to the $q$--deformed $k$--random--to--random family.


We take a similar approach for the $k$--star transposition shuffle. Its
transition matrix can be written in $\mathbb{C}[S_n]$ as a linear combination
of the identity and the Jucys--Murphy elements
$J_{n-k+1},\ldots,J_n$ \cite{arfaee2025shuffling}. These elements have
$q$--analogues $J_m(q)$ in $\mathcal{H}_q(S_n)$
\cite{Murphy1992, Ram1997}. Replacing the classical Jucys--Murphy elements by
their Hecke analogs leads to the $q$--deformed $k$--star transposition
shuffle studied in this paper.


\subsection{Main results}\label{sec:mainresults}

We consider the following three
examples. Throughout the paper, we fix $q>1$ and write $\theta=q^{-1}$.
We denote the short systematic scan by $K$ and the
$q$--random--to--random shuffle by $R$; their precise definitions are
recalled in Section~\ref{sec:eigenvalues}.

\emph{The $q$--deformed $k$--star transposition shuffle.}
Our third chain is obtained from the Hecke Jucys--Murphy elements. For
$1\leq k\leq n-1$, define
\begin{equation}\label{eq:Pdef}
P
=
\frac{1}{[n]_q}I
+
\frac{q[n-1]_q}
     {[n]_q\sum_{l=1}^{k}[n-l]_q}
\sum_{l=1}^{k}J_{n+1-l}(q),
\end{equation}
where $J_m(q)$ are the Jucys--Murphy elements of
$\mathcal{H}_q(S_n)$ defined in Section~\ref{sec:heckejm}.
The elements $J_1(q),\ldots,J_n(q)$ commute and have a common eigenbasis, so
the spectrum of $P$ can be obtained from their joint eigenvalues. 

The connection between these chains already appears before taking any limit.
For $k=1$, Proposition~\ref{prop:k1scan} shows that both $K$ and $P$ are affine
functions of the same Jucys--Murphy element $J_n(q)$. In particular, the
agreement between their mixing times is built into the operators themselves:
the same number of steps is enough for the two chains.

This is very different from what happens at $q=1$. The short systematic scan
then becomes trivial. Indeed, every Metropolis move is accepted, so the forward
and backward sweeps cancel and $K$ is the identity. The other two chains do not
degenerate: they become the usual random--to--random and $k$--star
transposition shuffles. Thus, at $q=1$ the three chains have quite different
behavior.


There is another point at which the three chains meet. As
$\theta\to0$, or equivalently $q\to\infty$, all three operators become the
same for every $k$; see Theorem~\ref{prop:theta0}. The limiting chain is
deterministic. Starting from a permutation, it performs exactly those swaps
which increase the number of inversions. Hence, the three chains, which are
different at $q=1$, have the same dynamics in the opposite limit.

We also use this limiting dynamics in the proof. Its orbit starting from the
identity can be described explicitly; see
Corollary~\ref{cor:theta0orbit}. Comparing the three chains with this
deterministic orbit gives the total variation lower bound for all of them.


We begin with the spectrum of the $q$--deformed $k$--star transposition
shuffle. Recall that $\mathrm{SYT}(\lambda)$ is the set of standard Young
tableaux of shape $\lambda$ and that
$d_\lambda=|\mathrm{SYT}(\lambda)|$.

For $r\in\mathbb{Z}$, we write
\[
    [r]_q=\frac{q^r-1}{q-1}.
\]
In particular, for $r\geq 1$,
\[
    [r]_q=1+q+\cdots+q^{r-1},
\]
and
\[
    [-r]_q=-q^{-r}[r]_q.
\]

\begin{theorem}\label{T-EVHeckeCombined}
Let $P$ be the operator in \eqref{eq:Pdef}. For $\lambda\vdash n$ and
$S\in\mathrm{SYT}(\lambda)$, let $(i_m(S),j_m(S))$ be the box containing
$m$ and write
\[
    c_m(S)=j_m(S)-i_m(S).
\]
Then the eigenvalue of $P$ indexed by $S$ is
\begin{equation}\label{eq:EVHeckeSYT}
\eig_q(S)
=
\frac{1}{[n]_q}
+
\frac{q[n-1]_q}
     {[n]_q\sum_{l=1}^{k}[n-l]_q}
\sum_{l=1}^{k}
\big[c_{n+1-l}(S)\big]_q .
\end{equation}

For a partition $\nu$, define its $q$--diagonal index by
\[
    \textup{Diag}_q(\nu)
    =
    \sum_{(i,j)\in\nu}[j-i]_q.
\]
If $\mu\vdash n-k$ and $\mu\subseteq\lambda$, then the same eigenvalues can be
indexed by pairs $(\lambda,\mu)$ and given by
\begin{equation}\label{eq:EVHeckeDiag}
\eig_q(\lambda,\mu)
=
\frac{1}{[n]_q}
+
\frac{q[n-1]_q}
     {[n]_q\sum_{l=1}^{k}[n-l]_q}
\bigl(
\textup{Diag}_q(\lambda)-\textup{Diag}_q(\mu)
\bigr).
\end{equation}
The eigenvalue corresponding to $(\lambda,\mu)$ has multiplicity
\[
    d_\mu d_{\lambda/\mu}.
\]
\end{theorem}

The proof is given in Section~\ref{sec:kstar-upper}. We use the simultaneous
eigenbasis of the Jucys--Murphy elements $J_m(q)$ and then group together
tableaux for which the entries $1,\ldots,n-k$ have the same shape $\mu$.

We next turn to mixing. Although the three chains have different transition operators, their mixing behaviors agree. The upper and lower bounds are proved separately for the three operators and combined at the end of Section~\ref{sec:l2lower} to prove Theorem~\ref{T-l2cutoff}.
\begin{theorem}[$\ell^{2}$ cutoff]\label{T-l2cutoff}
Fix $q>1$, and let $X$ be one of the following three chains: the
$q$--deformed $k$--star transposition shuffle $P$, with
$1\leq k\leq n-1$, the $q$--random--to--random shuffle $R$, or the short
systematic scan $K$. For $c>0$, set
\[
    t^\pm
    =
    \frac12\bigl(n+\log_q n\pm c\bigr).
\]
Starting from the identity,
\[
\lim_{c\to\infty}\lim_{n\to\infty}
\left\|
    \frac{X_{\mathrm{id}}^{\,t^+}}{\pi}-1
\right\|_2
=0,
\]
whereas, for all sufficiently large $n$,
\[
\left\|
    \frac{X_{\mathrm{id}}^{\,t^-}}{\pi}-1
\right\|_2^2
\geq
\frac{q^{c-1}}{4}.
\]
Hence each of the three chains exhibits $\ell^2$ cutoff at
\[
    \frac12\bigl(n+\log_q n\bigr)
\]
with a window of order $1$.
\end{theorem}

The mixing results stated above are for the chain starting at the identity.
For the upper bounds, this does not entail a loss: the identity is in fact the worst
starting state in $\ell^2$. This holds more generally for reversible Markov
chains arising from the Iwahori--Hecke algebra.

\begin{theorem}[The identity is the worst starting state]\label{T-worst}
Let $W$ be a finite Coxeter group, and let $X$ be a Markov chain on $W$ given
by left multiplication by an element of its Iwahori--Hecke algebra. Suppose
that $X$ is reversible with respect to $\pi$. Then, for every $t\geq0$,
\[
    \max_{x\in W}
    \left\|
        \frac{X_x^{\,t}}{\pi}-1
    \right\|_2^2
    =
    \left\|
        \frac{X_{\mathrm{id}}^{\,t}}{\pi}-1
    \right\|_2^2.
\]
In particular, the $\ell^2$ upper bound proved from the identity holds
uniformly over all starting states.
\end{theorem}

Theorem~\ref{T-worst} is proved in
Section~\ref{sec:kstar-upper}; no information about the spectrum of $X$ is
needed. The proof uses only the Hecke-algebra structure, together with
reversibility and the fact that the normalized generators
$\widetilde T_{s_i}$ are stochastic.

The upper bound also gives a total variation upper bound by (\ref{eq:tvl2}).
The lower bound in total variation requires a different argument and is proved in Section~\ref{sec:l1lower}.

\begin{theorem}[Total variation cutoff]\label{T-tv}
Fix $q>1$, and let $X$ be as in Theorem~\ref{T-l2cutoff}. Starting from the
identity, for every integer $t$ such that $n-2t\geq2$,
\[
    \bigl\|X_{\mathrm{id}}^{\,t}-\pi\bigr\|_{\mathrm{TV}}
    \geq
    1-
    \frac{n}
    {(q-1)\binom{n-2t}{2}}.
\]
In particular, if $\alpha>0$, then
\[
    \lim_{n\to\infty}
    d_n\left(\frac n2-\alpha\sqrt n\right)
    \geq
    1-\frac{1}{2(q-1)\alpha^2},
\]
where
\[
    d_n(t)=\max_x\|X_x^t-\pi\|_{\mathrm{TV}}.
\]
On the other hand, the $\ell^2$ upper bound gives
\[
    d_n\left(\frac n2+\alpha\sqrt n\right)\longrightarrow0
    \qquad\text{for every }\alpha>0.
\]
Consequently,
\[
\lim_{\alpha\to\infty}
\lim_{n\to\infty}
d_n\left(\frac n2-\alpha\sqrt n\right)=1,
\qquad
\lim_{\alpha\to\infty}
\lim_{n\to\infty}
d_n\left(\frac n2+\alpha\sqrt n\right)=0.
\]
Thus the three chains exhibit the total variation cutoff at $\frac{n}{2}$, with a window
of at most order $\sqrt n$.
\end{theorem}

For the upper bound, we keep the generic degree $t_\lambda$
(Definition~\ref{def:genericdegree}) in the spectral sum and estimate it
directly. This avoids replacing it by the bound
\[
    t_\lambda
    \leq
    \theta^{\binom{\lambda_1}{2}-\binom n2}d_\lambda
\]
from Lemma~7.2(a) of \cite{DiaconisRam2000}. For the total variation lower
bound, we instead use the degeneration at $\theta=0$. The limiting dynamics
is deterministic, so the required estimate reduces to counting
non-inversions along its orbit.


The scale $\sqrt n$ appears only in this lower-bound argument. The upper bound
already gives mixing at
\[
    \frac n2+\frac12\log_q n+O(1).
\]
We therefore do not expect $\sqrt n$ to be the correct total variation
window. In Conjecture~\ref{conj:tvcutoff}, we conjecture cutoff at
\[
    \frac12\bigl(n+\log_q n\bigr)
\]
with a window of order $1$.

\subsection{Organization}\label{sec:organisation}

The paper is organized as follows. In Section~\ref{sec:background}, we collect the background that will be used throughout the paper. We recall Metropolis chains on Coxeter groups and their interpretation in the Iwahori--Hecke algebra, the $\ell^2$ formula of Diaconis and Ram, the necessary representation theory of $\mathcal{H}_q(S_n)$, and the Jucys--Murphy elements. We also introduce the three chains and the Demazure product that will be used in the total variation lower bound. Section~\ref{sec:boundeigen} contains the spectral part of the paper. We diagonalize the $q$--deformed $k$--star transposition shuffle, bound its eigenvalues, and prove the $\ell^2$ upper bound. We then obtain the corresponding upper bounds for the $q$--random--to--random shuffle and the short systematic scan. Section~\ref{sec:l2lower} gives the matching $\ell^2$ lower bounds for the three chains, completing the proof of the $\ell^2$ cutoff result. In Section~\ref{sec:theta0}, we study the three chains in the limit $\theta=0$ and show that they all reduce to the same deterministic dynamics. Finally, in Section~\ref{sec:l1lower}, we use this limiting dynamics to obtain the total variation lower bound and complete the proof of total variation cutoff.


\section{Background and Preliminaries}\label{sec:background}

\subsection{Metropolis chains on Coxeter groups}\label{sec:metropolis}

We first recall the Metropolis chains on Coxeter groups introduced in
\cite{DiaconisRam2000}. Let $W$ be a finite Coxeter group generated by the
simple reflections $s_1,\ldots,s_r$, and let $\ell$ denote the length function
on $W$. Fix $0<\theta\leq 1$ and set $q=\theta^{-1}$. We consider the
probability measure
\begin{equation}\label{eq:stationary}
\pi(w)
=
\frac{q^{\ell(w)}}{P_W(q)},
\qquad
P_W(q)=\sum_{w\in W}q^{\ell(w)}.
\end{equation}
The polynomial $P_W(q)$ is the Poincar\'{e} polynomial of $W$. If
$d_1,\ldots,d_r$ are the degrees of $W$, then

$$
    P_W(q)
    =
    \prod_{i=1}^r\frac{q^{d_i}-1}{q-1}.
$$

For $W=S_n$, the degrees are $2,\ldots,n$, and therefore

$$
    P_{S_n}(q)=[n]_q!.
$$

For each simple reflection $s_i$, define the Markov chain $K_i$ by

$$
K_i(x,y)
=
\begin{cases}
1, & y=s_i x \text{ and } \ell(s_i x)>\ell(x),\\
\theta, & y=s_i x \text{ and } \ell(s_i x)<\ell(x),\\
1-\theta, & y=x \text{ and } \ell(s_i x)<\ell(x),\\
0, & \text{otherwise.}
\end{cases}
$$

In other words, starting from $x$, we multiply on the left by $s_i$. If this
increases the length, the move is always made. If it decreases the length, the
move is made with probability $\theta$, and otherwise the chain stays at $x$.

The chains that we consider later are obtained from the $K_i$'s by taking
products and convex combinations. We will use the following facts.

\begin{lemma}\label{lem:reversible}
\begin{enumerate}
\item Each $K_i$ is reversible with respect to $\pi$.

\item Every product $K_{i_1}\cdots K_{i_m}$ has stationary distribution
$\pi$. The same is true for any convex combination of such products.

\item If $i_1,\ldots,i_m$ is a palindrome, then
$K_{i_1}\cdots K_{i_m}$ is reversible with respect to $\pi$.
A convex combination of palindromic products is also reversible.
More generally, an operator of the form $A^*A$ is reversible.
\end{enumerate}
\end{lemma}

\begin{proof}
For part (1), fix $x\in W$ and suppose first that
$\ell(s_i x)=\ell(x)+1$. By \eqref{eq:stationary},

$$
    \pi(s_i x)=q\,\pi(x).
$$

On the other hand, by the definition of $K_i$,

$$
    K_i(x,s_i x)=1
    \qquad\text{and}\qquad
    K_i(s_i x,x)=\theta.
$$

Since $\theta=q^{-1}$, we get

$$
    \pi(x)K_i(x,s_i x)
    =\pi(x)
    =q\theta\,\pi(x)
    =\pi(s_i x)K_i(s_i x,x).
$$

The case $\ell(s_i x)=\ell(x)-1$ is the same calculation with $x$ and
$s_i x$ interchanged. Therefore $K_i$ satisfies detailed balance with
respect to $\pi$, and hence it is reversible.

For part (2), part (1) gives $\pi K_i=\pi$ for every $i$. Applying this
successively, we obtain

$$
    \pi K_{i_1}K_{i_2}\cdots K_{i_m}
    =
    \pi K_{i_2}\cdots K_{i_m}
    =\cdots=\pi.
$$

Thus every product of the $K_i$'s has stationary distribution $\pi$.
If

$$
    K=\sum_j a_j K^{(j)},
    \qquad a_j\geq0,\qquad \sum_j a_j=1,
$$

where each $K^{(j)}$ is such a product, then

$$
    \pi K
    =
    \sum_j a_j\pi K^{(j)}
    =
    \sum_j a_j\pi
    =
    \pi.
$$

This proves the statement for convex combinations.

For part (3), recall that a Markov chain is reversible with respect to
$\pi$ if and only if its transition operator is self-adjoint on
$L^2(\pi)$. By part (1), $K_i^*=K_i$ for every $i$, and therefore

$$
    \bigl(K_{i_1}K_{i_2}\cdots K_{i_m}\bigr)^*
    =
    K_{i_m}\cdots K_{i_2}K_{i_1}.
$$

If $(i_1,\ldots,i_m)$ is a palindrome, the right-hand side is equal to
$K_{i_1}\cdots K_{i_m}$, so the product is reversible. The same argument
shows that a real convex combination of palindromic products is
self-adjoint. Finally, for any operator $A$,

$$
    (A^*A)^*=A^*A,
$$

so whenever $A^*A$ is a Markov operator, it is also reversible with
respect to $\pi$.
\end{proof}

We now apply Lemma~\ref{lem:reversible} to the three chains considered in this
paper. For the short systematic scan, one step is

$$
    K_{n-1}\cdots K_2K_1^2K_2\cdots K_{n-1},
$$

which is a palindromic product, and hence it is reversible by
Lemma~\ref{lem:reversible}(3).

For the $q$--deformed $k$--star transposition shuffle, recall that for
$1\leq i<m$,

$$
    (i\,m)
    =
    s_i s_{i+1}\cdots s_{m-2}s_{m-1}
    s_{m-2}\cdots s_{i+1}s_i.
$$

Thus the element $\widetilde T_{(i,m)}$ corresponds to a palindromic
product of the elementary Metropolis chains. Since $J_m(q)$ is a linear
combination of these elements, the operator $P$ is a convex combination
of palindromic products and is therefore reversible.

Finally, the $q$--random--to--random shuffle can be written as

$$
    R
    =
    \frac{1}{[n]_q^2}\,
    \mathcal B_n^*(q)\mathcal B_n(q).
$$

Since $\mathcal B_n^*(q)$ is the adjoint of $\mathcal B_n(q)$, this is of
the form $A^*A$. Lemma~\ref{lem:reversible}(3) therefore also gives the
reversibility of $R$.


\subsection{Distances, mixing time}\label{sec:distances}

Let $\mu$ and $\pi$ be probability measures on a finite set $X$ with $\pi(x)>0$
for every $x$. We will use two distances to measure convergence to $\pi$: total
variation and $\ell^{2}(\pi)$. They are given by
$$
  \|\mu-\pi\|_{TV}
  \;=\; \max_{A \subseteq X}\bigl|\mu(A)-\pi(A)\bigr|
  \;=\; \tfrac12\sum_{x \in X}\bigl|\mu(x)-\pi(x)\bigr|,
$$
$$
  \Bigl\|\frac{\mu}{\pi}-1\Bigr\|_{2}^{2}
  \;=\; \sum_{x \in X}\pi(x)\Bigl(\frac{\mu(x)}{\pi(x)}-1\Bigr)^{\!2}
  \;=\; \sum_{x \in X}\frac{\mu(x)^{2}}{\pi(x)} \;-\; 1 .
$$
The two distances are related by Cauchy--Schwarz:
\begin{equation}\label{eq:tvl2}
  2\,\|\mu-\pi\|_{TV} \;\le\; \Bigl\|\frac{\mu}{\pi}-1\Bigr\|_{2}.
\end{equation}
Thus every $\ell^{2}$ upper bound also gives an upper bound in total variation.
For lower bounds this is no longer true, and later we will need a separate
argument in total variation.

Writing $K^{t}_{x}$ for the law of a chain $K$ after $t$ steps when it is
started at $x$, we define the total variation mixing time by
$$
  t_{\mathrm{mix}}(\varepsilon)
  \;=\; \min\Bigl\{\,t \;:\; \max_{x \in X}
        \bigl\|K^{t}_{x}-\pi\bigr\|_{TV} \le \varepsilon \Bigr\}.
$$
Since the maximum is taken over all starting states, a lower bound from one
particular starting state is already enough to give a lower bound on the mixing
time. A family of chains indexed by $n$ exhibits a \emph{cutoff} at $t_n$ if,
for every $\varepsilon\in(0,1)$, the distance tends to $1$ at time
$(1-\varepsilon)t_n$ and to $0$ at time $(1+\varepsilon)t_n$. The definition
of $\ell^{2}$ cutoff is the same, with
$\|\frac{K^{t}_{x}}{\pi(x)}-1\|_{2}$ in place of $\|K^{t}_{x}-\pi\|_{TV}$.

For the chains in this paper, the $\ell^{2}$ distance can be written in terms
of irreducible characters of the Iwahori--Hecke algebra. This is the formula
that we use throughout the mixing-time arguments. If $K$ is a reversible
Markov chain on $W$ given by left multiplication by an element of the
Iwahori--Hecke algebra $H$, then the $\ell^{2}$ distance decomposes over the
irreducible $H$--modules. The following result is Proposition~4.8 of
\cite{DiaconisRam2000}.

\begin{theorem}[{\cite[Proposition~4.8]{DiaconisRam2000}}]
\label{thm:DRexpansion}
Let $H$ be the Iwahori--Hecke algebra corresponding to a finite real
reflection group $W$.  Let $K$ be a reversible Markov chain on $W$ with
stationary distribution $\pi$ determined
by left multiplication by an element of $H$ (also denoted by $K$).
Let $K_x^{t}$ denote the Markov chain started at $x$ after $t$ steps.
Then

\begin{enumerate}
\item
\begin{equation}\label{eq:l2norm}
\big\| \frac{K_x^{t}}{\pi(x)} -1\big\|_{2}^2
=q^{-2\ell(x)}
\sum_{\lambda\ne \mathbf{1}}
t_\lambda\,
\chi^\lambda_{H}\!\big(T_{x^{-1}}K^{2t}T_x\big),
\end{equation}

\item
\begin{equation}
\sum_{x\in W}\pi(x)\,
\big\| \frac{K_x^{t}}{\pi(x)} -1\big\|_{2}^2
=
\sum_{\lambda\ne \mathbf{1}}
d_\lambda\,
\chi^\lambda_{H}\!\big(K^{2t}\big).
\end{equation}
\end{enumerate}

\end{theorem}

Here $\chi^{\lambda}_{H}$ is the character of the irreducible module
$S^{\lambda}_q$, $d_\lambda=\dim S^{\lambda}_q$, and $t_\lambda$ is its
\emph{generic degree} (Definition~\ref{def:genericdegree}). The term
$\mathbf{1}$ denotes the trivial representation; for $S_n$ this is
$\lambda=(n)$. Its contribution is the stationary part, so it is removed from
the two sums above.

There are two parts of the theorem that we will use repeatedly. First, the
character on the right-hand side can be written as a sum of powers of the
eigenvalues. Since the three chains are diagonal in the seminormal basis, their characters
can be expressed in terms of their eigenvalues using \eqref{eq:charsum}.
Thus the mixing problem reduces to controlling the resulting spectral sums. This is what we do in Section~\ref{sec:boundeigen}.

The second point is that the two identities in
Theorem~\ref{thm:DRexpansion} put different weights on the irreducible
representations. Part~(1), which gives the distance from a fixed starting
state, contains the generic degree $t_\lambda$. Part~(2), which averages over
the starting state with respect to $\pi$, contains only the dimension
$d_\lambda$. When $q>1$, these quantities can be very different. 



We will mostly use part~(1) with the chain started at the identity. Since
$\ell(\mathrm{id})=0$, the prefactor is $1$ and the two conjugating terms
disappear. In this case
$$
  \Bigl\|\frac{K^{t}_{\mathrm{id}}}{\pi}-1\Bigr\|_{2}^{2}
  \;=\; \sum_{\lambda \ne (n)}
  t_\lambda\,\chi^{\lambda}_{H}\bigl(K^{2t}\bigr),
$$
which is the form used in Sections~\ref{sec:boundeigen} and
\ref{sec:l2lower}. Finally, when $q=1$ we have $t_\lambda=d_\lambda$, and the
two formulas reduce to the classical Fourier--analytic formulas for random
walks on the symmetric group.


\subsection{The Iwahori--Hecke Algebra and Jucys--Murphy Elements}
\label{sec:heckejm}

We recall the definition of the Iwahori--Hecke algebra and the connection with the
Metropolis chains from the previous section. We then introduce the Hecke
Jucys--Murphy elements, which will be used to diagonalize the $q$--deformed
$k$--star transposition shuffle.

\begin{definition}[Type $A$ Iwahori--Hecke algebra]\label{def:hecke}
For $q\in\mathbb C$, the \emph{Type~$A$ Iwahori--Hecke algebra}
$\mathcal H_q(S_n)$ is the associative $\mathbb C$--algebra generated by
\[
T_{s_1},T_{s_2},\ldots,T_{s_{n-1}},
\]
with relations
\begin{enumerate}
\item $T_{s_i}^2=(q-1)T_{s_i}+q, \quad 1\le i\le n-1$;
\item $T_{s_i}T_{s_j}=T_{s_j}T_{s_i}, \quad|i-j|\ge2$;
\item
$
T_{s_i}T_{s_{i+1}}T_{s_i}
=
T_{s_{i+1}}T_{s_i}T_{s_{i+1}},
\quad 1\le i\le n-2.
$
\end{enumerate}
If $w=s_{i_1}\cdots s_{i_r}$ is a reduced expression, set
\[
T_w=T_{s_{i_1}}\cdots T_{s_{i_r}}.
\]
The braid relations imply that $T_w$ does not depend on the choice of reduced
expression, and $\{T_w\}_{w\in S_n}$ is a basis of $\mathcal H_q(S_n)$.
At $q=1$ the quadratic relation becomes $T_{s_i}^2=1$, so
$
\mathcal H_1(S_n)=\mathbb C[S_n].
$
\end{definition}

To see the connection with the Metropolis chains, it is convenient to use the
normalized basis
\[
\widetilde T_w=q^{-\ell(w)}T_w.
\]
In particular, $\widetilde T_i=q^{-1}T_i$. The following result says that left
multiplication by $\widetilde T_i$ is exactly one Metropolis update.

\begin{theorem}[{\cite[Theorem~4.3]{DiaconisRam2000}}]\label{thm:metropolis-hecke}
Let $W$ be a finite Coxeter group and let $H$ be its Iwahori--Hecke algebra with
basis $\{T_w\}_{w\in W}$. Set
\[
q=\theta^{-1},\qquad
\widetilde T_i=\frac{T_i}{q},\qquad
\widetilde T_w=q^{-\ell(w)}T_w.
\]
Then the Metropolis chain $K_i$ is the matrix of left multiplication by
$\widetilde T_i$ with respect to the basis
$\{\widetilde T_w\}_{w\in W}$. More precisely,
\[
\widetilde T_i\widetilde T_w=
\begin{cases}
\widetilde T_{s_iw},
& \ell(s_iw)>\ell(w),\\[1mm]
(1-\theta)\widetilde T_w+\theta\widetilde T_{s_iw},
& \ell(s_iw)<\ell(w).
\end{cases}
\]
\end{theorem}

Thus the products of Metropolis moves that appeared in the previous section can
be regarded as products in the Hecke algebra. We will use this point of view for
all three chains.

We next recall the Jucys--Murphy elements. In the group algebra
$\mathbb C[S_n]$ they are
\[
J_m=\sum_{i=1}^{m-1}(i\,m),
\qquad 2\le m\le n.
\]
Their Hecke analogues are
\[
J_m(q)=\sum_{i=1}^{m-1}q^{i-m}T_{(i\,m)},
\quad 2\le m\le n,
\quad \text{where} \quad
T_{(i\,m)}
:=
T_iT_{i+1}\cdots T_{m-2}T_{m-1}
T_{m-2}\cdots T_{i+1}T_i.
\]
Since
$
\ell((i\,m))=2(m-i)-1,
$
this can also be written in the normalized basis as
\[
J_m(q)
=
\sum_{i=1}^{m-1}
q^{m-1-i}\widetilde T_{(i\,m)},
\qquad 2\le m\le n.
\]
Setting $q=1$ gives back the usual Jucys--Murphy element $J_m$. The elements $J_2(q),\ldots,J_n(q)$ commute. On each irreducible
$\mathcal H_q(S_n)$--module they are diagonal in the seminormal basis indexed by
standard Young tableaux. If $S$ is a standard Young tableau and the entry $m$ is
in row $i_m(S)$ and column $j_m(S)$, then
\[
J_m(q)v_S
=
[j_m(S)-i_m(S)]_q\,v_S.
\]


\subsection{Partitions and tableaux}\label{sec:tableaux}

We use the notation for partitions and tableaux from \cite{Sagan2001}. A
\emph{partition} of $n$, written $\lambda\vdash n$, is a sequence
\[
    \lambda=(\lambda_1\geq\lambda_2\geq\cdots\geq\lambda_m>0)
\]
with $\sum_i\lambda_i=n$. We identify $\lambda$ with its Young diagram, whose
boxes are the pairs $(i,j)$ with $1\leq i\leq m$ and $1\leq j\leq\lambda_i$.
The conjugate partition $\lambda'$ is obtained by transposing the diagram, so
\[
    \lambda'_j=\#\{i:\lambda_i\geq j\}.
\]
In particular, $\lambda'_1=m$ is the number of rows of $\lambda$.

For a box $b=(i,j)\in\lambda$, its \emph{content} and \emph{hook length} are
\[
    c(b)=j-i,
    \qquad
    h_b=\lambda_i-j+\lambda'_j-i+1.
\]
We will also use
\[
    n(\lambda)=\sum_{i\geq1}(i-1)\lambda_i
              =\sum_{b\in\lambda}(\mathrm{row}(b)-1),
    \qquad
    n(\lambda')=\sum_{i\geq1}\binom{\lambda_i}{2},
\]
and the identity
\[
    \sum_{b\in\lambda}h_b=n(\lambda)+n(\lambda')+n.
\]

A \emph{standard Young tableau} of shape $\lambda$ is a filling of the boxes of
$\lambda$ with $1,\ldots,n$, each used once, such that the entries increase
along rows and down columns. We write $\mathrm{SYT}(\lambda)$ for the set of
standard Young tableaux of shape $\lambda$ and
\[
    d_\lambda=|\mathrm{SYT}(\lambda)|.
\]
The hook length formula gives
\[
    d_\lambda=\frac{n!}{\prod_{b\in\lambda}h_b},
\]
and, by the Robinson--Schensted correspondence,
\[
    \sum_{\lambda\vdash n}d_\lambda^2=n!.
\]
If $S\in\mathrm{SYT}(\lambda)$, we write $S(i,j)$ for the entry in the box
$(i,j)$. We also write $i_m(S)$ and $j_m(S)$ for the row and column containing
$m$, and
\[
    c_m(S)=j_m(S)-i_m(S)
\]
for the content of that box.

We will also need skew shapes. If $\mu\subseteq\lambda$, meaning that
$\mu_i\leq\lambda_i$ for every $i$, then $\lambda/\mu$ consists of the boxes
of $\lambda$ that are not in $\mu$. Let $d_{\lambda/\mu}$ denote the number of
standard fillings of this skew shape. If we restrict a tableau of shape
$\lambda$ to the entries $1,\ldots,n-k$, their boxes form a Young diagram
$\mu\vdash n-k$. This gives the branching identity
\begin{equation}\label{eq:branching}
  \sum_{\substack{\mu \subseteq \lambda \\ |\mu| = n-k}}
  d_\mu\, d_{\lambda/\mu} \;=\; d_\lambda .
\end{equation}
This is the identity that gives the multiplicities in
Theorem~\ref{T-EVHeckeCombined}. A skew shape $\lambda/\mu$ is called a
\emph{horizontal strip} if it contains at most one box in each column.

Finally, we use two particular tableaux of shape $\lambda$. Let
$T_{\lambda^{\downarrow}}$ be obtained by filling the first column from top to
bottom, then the second column, and so on. Similarly,
$T_{\lambda^{\to}}$ is obtained by filling the first row from left to right,
then the second row, and so on. These are the two superstandard tableaux. They
will appear in Lemma~\ref{lem:insertion}, where they give the two extreme ways
of placing the largest entries in a tableau of fixed shape.


\subsection{Representations of \texorpdfstring{$\mathcal{H}_q(S_n)$}{H\_q(S\_n)} and generic degrees}
\label{sec:reps}

For $q>0$ not a root of unity, the algebra $\mathcal{H}_q(S_n)$ is split
semisimple. Its irreducible modules are indexed by partitions $\lambda\vdash n$,
and we denote the module corresponding to $\lambda$ by $S_q^\lambda$. Its
dimension is
\[
    \dim S_q^\lambda=d_\lambda;
\]
see \cite{DipperJames1986}. We use the seminormal basis
\[
    \{v_S:S\in\mathrm{SYT}(\lambda)\}
\]
of \cite{Murphy1992, Ram1997}. As recalled in Section~\ref{sec:heckejm}, the
Jucys--Murphy elements are diagonal in this basis. When $q=1$, this becomes the
usual Young orthogonal form for the irreducible $S_n$--module $S^\lambda$, and
the eigenvalues of the $J_k$ are the ordinary contents.

We write $\chi_\lambda$, or $\chi_H^\lambda$ when we want to emphasize the
Hecke algebra, for the character of $S_q^\lambda$. If an element
$A\in\mathcal{H}_q(S_n)$ is diagonal in the seminormal basis, with eigenvalue
$\mathrm{eig}_q(S)$ on $v_S$, then
\begin{equation}\label{eq:charsum}
  \chi_\lambda\bigl(A^{\,r}\bigr)
  \;=\; \sum_{S \in \mathrm{SYT}(\lambda)} \mathrm{eig}_q(S)^{\,r} ,
\end{equation}
which is the form of the character that we use in
Theorem~\ref{thm:DRexpansion}.

\begin{definition}[generic degree]\label{def:genericdegree}
For $\lambda \vdash n$ the \emph{generic degree} of $S^{\lambda}_q$ is
\[
  t_\lambda
  \;=\; q^{\,n(\lambda)}\,
        \frac{\prod_{l=1}^{n}\bigl(q^{\,l}-1\bigr)}
             {\prod_{b\in\lambda}\bigl(q^{\,h_b}-1\bigr)}
  \;=\; q^{\binom{n}{2}-n(\lambda')}\,
        \frac{\prod_{l=1}^{n}\bigl(1-q^{-l}\bigr)}
             {\prod_{b\in\lambda}\bigl(1-q^{-h_b}\bigr)} ,
\]
the weight attached to $S^{\lambda}_q$ in \textup{(7.1)} of
\cite{DiaconisRam2000}.
\end{definition}

The first expression can also be written in terms of $q$--integers. Cancelling
$(q-1)^n$ from the numerator and denominator gives
\[
    t_\lambda
    =q^{n(\lambda)}
      \frac{\prod_{l=1}^n[l]_q}
           {\prod_{b\in\lambda}[h_b]_q}.
\]

The distinction between $t_\lambda$ and $d_\lambda$ is important for the
mixing estimates. In part~(1) of Theorem~\ref{thm:DRexpansion}, which gives the
$\ell^2$ distance from a fixed starting state, the irreducible indexed by
$\lambda$ is weighted by $t_\lambda$. In part~(2), where the starting state is
averaged with respect to $\pi$, the weight is $d_\lambda$. These two quantities
can have very different sizes when $q>1$. For example, if
$\lambda=(n-1,1)$, then
\[
    n(\lambda')=\binom{n-1}{2},
    \qquad
    t_{(n-1,1)}=q[n-1]_q,
\qquad
    d_{(n-1,1)}=n-1.
\]
Thus, for fixed $q>1$, $t_{(n-1,1)}$ is of order $q^n$, whereas the dimension
is only of order $n$. This difference is one of the reasons that the chain
started at the identity mixes on a different scale from the $\pi$--average
over starting states.


\subsection{Eigenvalues}\label{sec:eigenvalues}

The eigenvalues of the $q$--deformed $k$--star transposition shuffle $P$
are given in Theorem~\ref{T-EVHeckeCombined}. We record here the corresponding formulas for the
short systematic scan and the $q$--random--to--random shuffle, using the
notation of Section~\ref{sec:heckejm}.

\begin{definition}\label{def:three}
Let $q>1$.
\begin{enumerate}

\item In the \emph{short systematic scan}, one sweeps through the adjacent pairs in one
direction and then back, applying the Metropolis move at each pair. Thus one scan is the
operator
\[
  K \;=\; \widetilde T_{s_{n-1}}\cdots\widetilde T_{s_2}\,
          \widetilde T_{s_1}^{2}\,
          \widetilde T_{s_2}\cdots\widetilde T_{s_{n-1}} .
\]
Diaconis and Ram \cite{DiaconisRam2000} computed its eigenvalues in the seminormal basis.
For $S\in\mathrm{SYT}(\lambda)$, the corresponding eigenvalue is
\begin{equation}\label{eq:scaneig}
  \eig_q(S) \;=\; q^{\,c_n(S)-n+1} ,
\end{equation}
where $c_n(S)$ is the content of the box containing $n$. This is Theorem~7.5(a) of
\cite{DiaconisRam2000}, with their eigenvalue
$\theta^{\,n-1-c_n(S)}$ written in terms of $q=\theta^{-1}$.

\item The \emph{$q$--random--to--random shuffle} is obtained by composing a
$q$--random--to--bottom move with a $q$--bottom--to--random move
\cite{AxelrodFreedBraunerChiangComminsLang2024}. In our notation the two operators are
\[
  \mathcal{B}^{*}_n(q) = \sum_{i=1}^{n} q^{\,n-i}\,
        \widetilde T_{s_i s_{i+1}\cdots s_{n-1}},
  \qquad
  \mathcal{B}_n(q) = \sum_{j=1}^{n} q^{\,n-j}\,
        \widetilde T_{s_{n-1}s_{n-2}\cdots s_j},
\]
where the terms with $i=n$ and $j=n$ are $\widetilde T_e$. The transition operator is
\[
  R \;=\; \frac{1}{[n]_q^{2}}\,\mathcal{B}^{*}_n(q)\,\mathcal{B}_n(q).
\]
Each of $\mathcal{B}^{*}_n(q)$ and $\mathcal{B}_n(q)$ has coefficient sum $[n]_q$, so
$R$ is a convex combination of basis elements.

The eigenvalues of $R$ were computed in
\cite{AxelrodFreedBraunerChiangComminsLang2024}. They are indexed by partitions
$\mu\subseteq\lambda$ for which $\lambda/\mu$ is a horizontal strip, and are given by
\begin{equation}\label{eq:r2reig}
  \eig_q(\lambda,\mu) \;=\; \frac{E_{\lambda/\mu}(q)}{[n]_q^{2}},
  \qquad
  E_{\lambda/\mu}(q) \;=\; q^{\,n}\sum_{b\in\lambda/\mu}\bigl[c(b)\bigr]_q
     \;+\; \sum_{r=|\mu|+1}^{n} q^{\,n-r}[r]_q .
\end{equation}

\end{enumerate}
\end{definition}





\subsection{The Demazure product and the \texorpdfstring{$\theta=0$}{theta=0} chain}
\label{sec:demazure}

Here we recall the definition of the Demazure product from \cite{li2024demazure} and \cite{Norton_1979}.
\begin{definition} \label{def:zerohecke} The Demazure product (also called the $0$-Hecke product or greedy product) on $S_n$ is defined as follows:

\[
 s_i \star  w =
\begin{cases}
s_i w, & \text{if } \ell(s_i w) > \ell(w),\\
w, & \text{if } \ell(s_i w) < \ell(w).
\end{cases}.
\]

This is also the same as taking the product of $\tilde T_i$ in $H(q)$ and letting $q$ tend to infinity. This provides $S_n$ with a monoidal structure. 
\end{definition}

We give a Markov chain interpretation of this product. Let $(X^t)$ be a reversible Markov chain on a Coxeter group $W$ determined by multiplication by an element $S = \sum_{w \in W} c_w\tilde T_w$ of the Iwahori-Hecke algebra $H(q)$. The process moves from $X^t$ to $X^{t+1}$ as follows: first, $\tilde T_w$ is chosen according to the weights $c_w$. Then we write $\tilde T_w = \tilde T_{s_{i_1}}\tilde T_{s_{i_2}}\ldots \tilde T_{s_{i_l}}$ as a reduced word and apply each $\tilde T_{s_{i_j}}$ according to $K_{i_j}$. Define a chain $(Z^t_X)$ by the same process, except that each $T_{s_{i_j}}$ is applied according to $K_{i_j}$, with $\theta = 0$. Thus the chain $Z^t$ is determined by multiplication by the same element $S$ in the $0$-Hecke algebra. This chain is not reversible, since it is impossible for the chain to move from higher length elements to lower length elements.

\begin{definition}\label{def:subword}
Given a word $Q = \sigma_1 \sigma_2 \ldots \sigma_m$, where each $\sigma_i \in \{s_1,\ldots,s_{n-1}\}$, we say $\alpha$ is a subword of $Q$ if $\alpha = \sigma_{i_1}\ldots \sigma_{i_l}$, where $i_j$ runs through a subsequence of $1,2,\ldots,m$.
\end{definition}
If $\alpha$ is a subword of $Q$, then $\ell(\alpha) \le \ell(Q)$. By Corollary 2.2.3 of \cite{bjorner2005combinatorics}, these definitions do not depend on the choice of reduced subword expression of $Q$.




\section{Bounding the eigenvalues} \label{sec:boundeigen}

\begin{lemma}[Basic properties of $q$--integers]\label{lem:qinteger-facts}
Let $q>0$, $q\neq1$, and define the $q$--integer by
$$
[a]_q = \frac{q^{\,a}-1}{q-1}.
$$
Then the following hold:
\begin{enumerate}[label=\roman*)]
\item (\emph{Monotonicity}) If $a\le b$, then $[a]_q \le [b]_q$ for all $q>0$.\\
Hence, the function $a\mapsto [a]_q$ is nondecreasing, and strictly increasing when $a<b$.

\item (\emph{Symmetry}) For any real $a$,
$$
[-a]_q = -\,q^{-a}\,[a]_q.
$$

\item (\emph{Additivity identity}) For all integers $a,b\ge0$,
$$
[a+b]_q = [a]_q + q^{\,a}\,[b]_q.
$$

\item (\emph{Classical limit}) As $q\to1$,
$$
\lim_{q\to1} [a]_q = a.
$$

\item (\emph{Upper Bound}) For all integers $a$ and $q \ge 1$,
$$
[a]_q \le a\cdot q^{a-1}
$$

\item\label{it:qintlead} (\emph{Infinite bias limit}) For every integer $a\ge1$, as
$q\to\infty$, equivalently as $\theta = q^{-1}\to 0$,
$$
  [a]_q \;=\; 1+q+\cdots+q^{\,a-1}
        \;=\; q^{\,a-1}\bigl(1+\theta+\cdots+\theta^{\,a-1}\bigr)
        \;=\; q^{\,a-1}\bigl(1+O(\theta)\bigr),
$$
so that $\lim_{q\to\infty} q^{\,1-a}\,[a]_q = 1$. We abbreviate this by writing
$[a]_q = q^{\,a-1}$ at $\theta = 0$.

\end{enumerate}
\end{lemma}

\begin{proof}
For an integer $a\ge1$, one has $[a]_q = 1 + q + \cdots + q^{a-1}$, 
which is nondecreasing in $a$ for $q>0$, with strict increase when $a<b$.
The remaining properties follow from direct algebraic manipulation of 
$[a]_q = \frac{q^{a}-1}{q-1}$.
\end{proof}


\subsection{Upper bound for the \texorpdfstring{$q$--deformed $k$--star}{q-deformed k-star} transposition shuffle}\label{sec:kstar-upper}

We begin by proving Theorem~\ref{T-EVHeckeCombined}, which diagonalizes $P$. The point is
that $P$ is an affine combination of Jucys--Murphy elements, and those commute and are
already diagonal in the seminormal basis.

\begin{proof}[Proof of Theorem~\ref{T-EVHeckeCombined}]
This is the $q$--deformed analogue of Theorem~1.1 and Lemma~3.1 of
\cite{arfaee2025shuffling}. The only change is that the classical
Jucys--Murphy elements are replaced by their Hecke analogues and the ordinary
contents are replaced by $q$--contents.

Indeed, in the seminormal basis,

$$
    J_m(q)v_S=[c_m(S)]_qv_S.
$$

Since $P$ is a linear combination of the identity and
$J_{n-k+1}(q),\ldots,J_n(q)$, its eigenvalue corresponding to
$S\in\mathrm{SYT}(\lambda)$ is

$$
  \eig_q(S)
  = \frac{1}{[n]_q}
    + \frac{q[n-1]_q}{[n]_q\sum_{l=1}^{k}[n-l]_q}
      \sum_{l=1}^{k}[c_{n+1-l}(S)]_q,
$$

which gives \eqref{eq:EVHeckeSYT}.

As in the proof of Theorem~1.1 of \cite{arfaee2025shuffling}, let $\mu$ be the
shape formed by the entries $1,\ldots,n-k$. Then the remaining boxes form
$\lambda/\mu$, and

$$
    \sum_{l=1}^{k}[c_{n+1-l}(S)]_q
    =
    \textup{Diag}_q(\lambda)-\textup{Diag}_q(\mu).
$$

This gives \eqref{eq:EVHeckeDiag}.
Indeed, for fixed $\mu\subseteq\lambda$, such a tableau is obtained
by choosing a standard tableau of shape $\mu$ on the entries
$1,\ldots,n-k$ and a standard skew tableau of shape $\lambda/\mu$
on the remaining $k$ entries. Thus there are
$d_\mu d_{\lambda/\mu}$ such tableaux.
\end{proof}

Throughout the rest of this subsection we abbreviate
$$
  \eig_q(S)
  = \frac{1}{[n]_q} + \frac{q[n-1]_q}{[n]_q}\cdot\frac{D_S}{\Sigma_k},
  \qquad
  D_S = \sum_{l=1}^{k}\bigl[c_{n+1-l}(S)\bigr]_q .
$$

\begin{lemma}[$q$-analogue of Lemma 4.1 of \cite{arfaee2025shuffling}]\label{lem:insertion}
Let $\lambda \vdash n$ and $S \in \mathrm{SYT}(\lambda)$. Then
$$
  \mathrm{eig}_q\bigl(T_{\lambda^{\to}}\bigr)\ \le\ \mathrm{eig}_q(S)\ \le\
  \mathrm{eig}_q\bigl(T_{\lambda^{\downarrow}}\bigr),
$$
where $T_{\lambda^{\downarrow}}$ and $T_{\lambda^{\to}}$ denote the
column-insertion and row-insertion tableaux of shape $\lambda$.
\end{lemma}

\begin{proof}
Exchanging entries $a<n-k+1\le b$ occupying boxes of contents $c_a,c_b$
changes the ordinary ($q=1$) diagonal index $D_S$ by $c_b-c_a$, and $D_S$
by $[c_b]_q-[c_a]_q$. Since $a\mapsto[a]_q$ is increasing, these two
differences always have the same sign. Hence every swap used in the proof of
Lemma 4.1 of \cite{arfaee2025shuffling} -- which never decreases $c_b-c_a$ -- also never decreases
$[c_b]_q-[c_a]_q$, so the same sequence of swaps transforms $S$ into
$T_{\lambda^{\downarrow}}$ without ever decreasing $D_S$, and likewise
into $T_{\lambda^{\to}}$ without ever increasing it.
\end{proof}

Lemma~\ref{lem:insertion} reduces the problem of bounding $\eig_q(S)$ over all
$S \in \mathrm{SYT}(\lambda)$ to the two extreme tableaux. We now turn that into explicit
bounds in terms of $\lambda_1$ and the number of rows of $\lambda$.

\begin{lemma}\label{lem:eigbounds}
Let $\lambda \vdash n$, let $S \in \mathrm{SYT}(\lambda)$, let $m$ denote the
number of rows of $\lambda$, let $j = n - \lambda_1$, and let
$1 \le k \le n-1$.
\begin{enumerate}
\item \(\displaystyle 
\bigl|\mathrm{eig}_q(S)\bigr|\ \le\ \,\frac{k}{q^j}.\)
\item \(\displaystyle \bigl|\mathrm{eig}_q(S)\bigr|\ \le\ \frac{n+1}{n} \frac{1}{q^{j}}\), if
$\lambda_1 > \frac{6n}{10}$ and $n$ is sufficiently
large.
\end{enumerate}
\end{lemma}

\begin{proof}
We use the abbreviation $\eig_q(S) = \frac{1}{[n]_q}+\frac{q[n-1]_q}{[n]_q}\cdot
\frac{D_S}{\Sigma_k}$ introduced above. Since $k \le n-1$, the $l=1$ summand of $\Sigma_k$
gives
\begin{equation}\label{eq:sigmalower}
  \Sigma_k \ \ge\ [n-1]_q .
\end{equation}

\emph{(1), positive part.} Every box of $\lambda$ has content at most
$\lambda_1 - 1$, so $D_S \le k[\lambda_1-1]_q$ and hence, by
\eqref{eq:sigmalower},
$$
   \frac{D_S}{\Sigma_k}  \le\ \frac{k[\lambda_1-1]_q}{[n-1]_q}.
$$
It follows that
$$
  \ \mathrm{eig}_q(S)
  \ \le\ \frac{1 + qk[\lambda_1-1]_q}{[n]_q}
  \ \le\ \frac{k\bigl(1 + q[\lambda_1-1]_q\bigr)}{[n]_q}
  \ =\ \frac{k[\lambda_1]_q}{[n]_q}
  \ \le\ \frac{k}{q^j},
$$
where we used $[\lambda_1]_q = 1 + q[\lambda_1-1]_q$. This gives the upper bound.

\emph{(1), negative part.} Every box of $\lambda$ has content at least 
$-j$, so $-D_S \le -k[-j]_q=\frac{k[j]_q}{q^j} $ and hence, by
\eqref{eq:sigmalower},
$$
   \frac{-D_S}{\Sigma_k}  \le\ \frac{k[j]_q}{q^j[n-1]_q}.
$$
It follows that
$$
  \ -\mathrm{eig}_q(S)
  \ \le\ \frac{-1 + \frac{qk}{q^j}[j]_q}{[n]_q}
  \ \le\ \frac{k q[j]_q}{q^j[n]_q}
  \ \le\ \frac{k}{q^{n-1}}
  \ \le\ \frac{k}{q^{j}},
$$
where we used $j \leq n-1$. This gives the lower bound.

\emph{(2), positive part.} By Lemma~\ref{lem:insertion}, it suffices to bound
$\mathrm{eig}_q\bigl(T_{\lambda^{\downarrow}}\bigr)$ from above. We consider
two cases, as in the proof of Lemma 4.5 of \cite{arfaee2025shuffling}.

\emph{Case 1: $k \le \lambda_1 - \lambda_2$.} The entries
$n-k+1, \ldots, n$ occupy the boxes
$(1,\lambda_1-k+1), \ldots, (1,\lambda_1)$ of $T_{\lambda^{\downarrow}}$, so
$$
  D_{T_{\lambda^{\downarrow}}} \;=\; \sum_{s=1}^{k}[\lambda_1 - s]_q .
$$
For $1 \le s$ and $\lambda_1 \le n$,
$$
  q^{\lambda_1-1} - q^{\lambda_1-s}\ \le\ q^{n-1} - q^{n-s}
  \qquad\Longleftrightarrow\qquad
  [\lambda_1-s]_q\,[n-1]_q\ \le\ [\lambda_1-1]_q\,[n-s]_q ,
$$
so that $D_{T_{\lambda^{\downarrow}}} \le
\frac{[\lambda_1-1]_q}{[n-1]_q}\,\Sigma_k$, and therefore
$$
  \mathrm{eig}_q\bigl(T_{\lambda^{\downarrow}}\bigr)
  \ \le\ \frac{1 + q[\lambda_1-1]_q}{[n]_q}
  \ =\ \frac{[\lambda_1]_q}{[n]_q}
  \ \le\ q^{-j}
  \ \le\ \frac{n+1}{n\,q^{\,j}} .
$$

\emph{Case 2: $k > \lambda_1 - \lambda_2$.} Moving a box of the skew shape
occupied by the entries $n-k+1,\ldots,n$ upwards or to the right increases its
content, and hence increases $D_{T_{\lambda^{\downarrow}}}$, since
$a \mapsto [a]_q$ is increasing. As in the case $q=1$, the maximum of
$\mathrm{eig}_q\bigl(T_{\lambda^{\downarrow}}\bigr)$ over shapes with first
row $\lambda_1$ is attained at $\lambda = (\lambda_1, n-\lambda_1)$. For this
shape, $n-k < 2(n-\lambda_1)$, so the entries $1, \ldots, n-k$ of
$T_{\lambda^{\downarrow}}$ occupy
$\mu = \bigl(\lceil\tfrac{n-k}{2}\rceil, \lfloor\tfrac{n-k}{2}\rfloor\bigr)$.
Using $\sum_{c=0}^{a-1}[c]_q = \frac{[a]_q - a}{q-1}$ and
$[c-1]_q = \frac{1}{q}\bigl([c]_q - 1\bigr)$, we compute
$$
  \frac{D_{T_{\lambda^{\downarrow}}}}{\Sigma_k}
  \;=\;
  \frac{[\lambda_1]_q - \bigl[\lceil\tfrac{n-k}{2}\rceil\bigr]_q
        + \tfrac{1}{q}\Bigl([n-\lambda_1]_q
        - \bigl[\lfloor\tfrac{n-k}{2}\rfloor\bigr]_q\Bigr) - k}
       {q^{\,n-k}[k]_q - k}.
$$
Since $\lambda_1 \ge \frac{6n}{10}$, we have
$k >  2\lambda_1 - n \ge \frac{n}{5}$, so for $n$
sufficiently large,
$$
  q^{\,n-k}[k]_q - k \;\ge\; \frac{q^{\,n}}{q-1}\bigl(1 - 2q^{\frac{n}{5}}\bigr),
  \qquad
  [\lambda_1]_q + \tfrac{1}{q}[n-\lambda_1]_q
  \;\le\; \frac{q^{\lambda_1}}{q-1}\bigl(1 + q^{\frac{n}{5}}\bigr),
$$
whence $\frac{D_{T_{\lambda^{\downarrow}}}}{\Sigma_k} \le
q^{-j}\bigl(1 + 4q^{\frac{n}{5}}\bigr)$. Since $q[n-1]_q < [n]_q$ and
$\frac{1}{[n]_q} \le q^{-j}\,q^{\,1-\frac{3n}{5}}$,
$$
  \mathrm{eig}_q\bigl(T_{\lambda^{\downarrow}}\bigr)
  \;\le\; q^{-j}\bigl(1 + 5q^{\frac{n}{5}}\bigr)
  \;\le\; \frac{n+1}{n\,q^{\,j}}
$$
for all sufficiently large $n$, depending on $q$.

\emph{(2), negative part.} As in the proof of the lower bound of (1),
$-D_S \le \frac{k[j]_q}{q^j}$. Since $k \le n$, we can combine this bound with \eqref{eq:sigmalower} to get
$$
  \mathrm{eig}_q(S)
  \;\ge\; -\,\frac{n\,[\,j\,]_q}{[n]_q}
  \;\ge\; -\,n\,q^{\,j-n}
  \;\ge\; -\,\frac{n+1}{n\,q^{\,j}},
$$
the last inequality holding because $n^2 q^{\,2j} \le q^{\,n}$ for
$j \le \frac{2n}{5}$ and $n$ sufficiently large.


\end{proof}

We need a few classical estimates to bound \eqref{eq:l2norm}. These upper bounds come from \cite{DiaconisRam2000} and \cite{Diaconis1981}.

\begin{lemma}\label{lem:bounds}
Let $d_\lambda$ and
$t_\lambda$ be as in Section~\ref{sec:tableaux} and
Definition~\ref{def:genericdegree}. Then
\begin{enumerate}
\item $\displaystyle
  \sum_{\substack{\lambda \,\vdash\, n \\ \lambda_1=n-j}} d_\lambda
  \;\le\; \frac{(en)^{j}}{\sqrt{j!}}\,;$
\item $\displaystyle
  \sum_{\substack{\lambda \,\vdash\, n \\ \lambda_1=n-j }} d_\lambda^{\,2}
  \;\le\; \frac{n^{2j}}{j!}\,;$
\item $\displaystyle
  t_\lambda \;\le\; \left(\frac{q^{\,n}}{q-1}\right)^{j} .$
\end{enumerate}
\end{lemma}

\begin{proof}
\emph{(1).} As in the proof of Lemma 7.2 of \cite{DiaconisRam2000}, a standard
tableau with a first row of length $n-j$ is determined by the set of $j$ entries
lying outside its first row, together with their arrangement. There are
$\binom{n}{j}$ choices for the set, and the arrangement is a standard tableau
on $j$ boxes, of which there are at most $\sqrt{j!}\,e^{\,j}$. Hence
$$
  \sum_{\substack{\lambda \,\vdash\, n \\ \lambda_1 = n-j}} d_\lambda
  \;\le\; \binom{n}{j}\sqrt{j!}\,e^{\,j}
  \;\le\; \frac{n^{j}}{j!}\,\sqrt{j!}\,e^{\,j}
  \;=\; \frac{(en)^{j}}{\sqrt{j!}} .
$$
 
\emph{(2).} A standard tableau of shape $\lambda$ with $\lambda_1 = n-j$ is
determined by the set of $j$ entries lying below its first row, of which there
are $\binom{n}{j}$ choices, together with a standard tableau of the shape
$\mu = (\lambda_2,\lambda_3,\dots) \vdash j$ that they form; the first row is
then forced. Hence $d_\lambda \le \binom{n}{j}d_\mu$, and squaring and summing
over $\mu \vdash j$ gives
$$
  \sum_{\substack{\lambda \,\vdash\, n \\ \lambda_1 = n-j}} d_\lambda^{\,2}
  \;\le\; \binom{n}{j}^{2}\sum_{\mu \,\vdash\, j} d_\mu^{\,2}
  \;=\; \binom{n}{j}^{2} j!
  \;\le\; \frac{n^{2j}}{j!} ,
$$
using $\sum_{\mu\vdash j} d_\mu^{2} = j!$ and $\binom{n}{j}\le \frac{n^{j}}{j!}$.
 
\emph{(3).} Write $h_b$ for the hook length of the box $b \in \lambda$ and
recall that $n(\lambda) = \sum_i (i-1)\lambda_i$,
$n(\lambda') = \sum_i \binom{\lambda_i}{2}$, and
$\sum_{b \in \lambda} h_b = n(\lambda)+n(\lambda')+n$. After cancelling the $n$
factors of $q-1$ common to the numerator and denominator and extracting the
leading powers of $q$, we obtain
$$
  t_\lambda
  \;=\; q^{\,n(\lambda)}\,
        \frac{\prod_{l=1}^{n}\bigl(q^{\,l}-1\bigr)}
             {\prod_{b\in\lambda}\bigl(q^{\,h_b}-1\bigr)}
  \;=\; q^{\binom{n}{2}-n(\lambda')}\,
        \frac{\prod_{l=1}^{n}\bigl(1-q^{-l}\bigr)}
             {\prod_{b\in\lambda}\bigl(1-q^{-h_b}\bigr)} .
$$
The hook lengths in the first row of $\lambda$ are $\lambda_1$ distinct
positive integers, so
$\prod_{b \in \text{row }1}\bigl(1-q^{-h_b}\bigr) \ge
 \prod_{c=1}^{\lambda_1}\bigl(1-q^{-c}\bigr)$; cancelling this against the
first $\lambda_1$ factors of the numerator leaves
$$
  t_\lambda
  \;\le\; q^{\binom{n}{2}-n(\lambda')}\,
          \frac{\prod_{l=\lambda_1+1}^{n}\bigl(1-q^{-l}\bigr)}
               {\prod_{b \notin \text{row }1}\bigl(1-q^{-h_b}\bigr)}
  \;\le\; q^{\binom{n}{2}-n(\lambda')}\bigl(1-q^{-1}\bigr)^{-j},
$$
since the numerator is at most $1$ and each of the $j$ boxes outside the first
row has $h_b \ge 1$. Now $n(\lambda') \ge \binom{\lambda_1}{2}$ and
$\binom{n}{2}-\binom{\lambda_1}{2} = nj-\binom{j+1}{2}$, so, using
$\bigl(1-q^{-1}\bigr)^{-j} = q^{\,j}(q-1)^{-j}$,
$$
  t_\lambda
  \;\le\; \frac{q^{\,nj-\binom{j+1}{2}+j}}{(q-1)^{j}}
  \;=\; \frac{q^{\,nj-\binom{j}{2}}}{(q-1)^{j}}
  \;\le\; \frac{q^{\,nj}}{(q-1)^{j}}
  \;=\; \left(\frac{q^{\,n}}{q-1}\right)^{j} .
$$
\end{proof}

We are now ready to prove the upper bound.

\begin{theorem}\label{thm:upper}
Let $q>1$ be fixed and let $1 \le k \le n-1$. As in \cite{DiaconisRam2000}, we
start the chain at the identity, which is the state with the smallest
$\pi$-probability. Let $c>0$ and

$$
  t \;=\; \tfrac12\bigl(n + \log_q n + c\bigr),
  \qquad
  A \;=\; \frac{e\,q^{-c}}{q-1} .
$$

Then, for $n$ large enough,

$$
  \Bigl\|\frac{P^{\,t}_{\mathrm{id}}}{\pi}-1\Bigr\|_2^{2}
  \;\le\;
  n\,q^{-\frac{8}{100}n^{2} + O(n\log n)}
  \;+\;
  e^{2}\sqrt{e^{2A^{2}}-1}\,.
$$

In particular,

$$
  \lim_{c\to\infty}\ \limsup_{n\to\infty}\
  \Bigl\|\frac{P^{\,t}_{\mathrm{id}}}{\pi}-1\Bigr\|_2 \;=\; 0,
$$

and hence $\frac12\bigl(n+\log_q n+c\bigr)$ steps are enough for the chain to
mix, both in $\ell^2$ and in total variation.
\end{theorem}

\begin{proof}
We use Proposition 4.8(a) of \cite{DiaconisRam2000} with $x=\mathrm{id}$. In
this case the factor $q^{-2\ell(x)}$ equals $1$. Since $P$ is diagonal on
$S^{\lambda}_q$ in the seminormal basis, we have
\begin{equation}\label{eq:master}
  \Bigl\|\frac{P^{\,t}_{\mathrm{id}}}{\pi}-1\Bigr\|_2^{2}
  \;=\; \sum_{\substack{\lambda \vdash n \\ \lambda \ne (n)}}
        t_\lambda\, \chi_\lambda\bigl(P^{2t}\bigr),
  \qquad
  \chi_\lambda\bigl(P^{2t}\bigr)
  \;=\; \sum_{S \in \mathrm{SYT}(\lambda)} \eig_q(S)^{2t}.
\end{equation}

Let $m$ be the number of rows of $\lambda$ and let $j = n-\lambda_1$. We
separate the sum in \eqref{eq:master} into the cases
$\lambda_1 \le \frac{6n}{10}$ and $\lambda_1 > \frac{6n}{10}$. We also use
that $\mathrm{SYT}(\lambda)$ has $d_\lambda$ elements, and therefore
$$\chi_\lambda(P^{2t}) \le d_\lambda \max_{S} |\eig_q(S)|^{2t}.$$

\emph{The case $\lambda_1 \le \frac{6n}{10}$.} Here
$j \ge \frac{4n}{10}$. By Lemma~\ref{lem:eigbounds}, we have
\begin{equation}\label{eq:z1eig}
\bigl|\eig_q(S)\bigr| \;\le\; \frac{k}{q^{j}}
\qquad\text{for every } S \in \mathrm{SYT}(\lambda).
\end{equation}





From the proof of Lemma~\ref{lem:bounds}(3), we have
$t_\lambda \le q^{nj-\binom{j}{2}}(q-1)^{-j}$, while
Lemma~\ref{lem:bounds}(1) gives
$\sum_{\lambda_1=n-j} d_\lambda \le (en)^{j}$. We group the terms according to
$j$. Taking $\log_q$ of the resulting summand gives

$$
  \Bigl(nj-\binom{j}{2}\Bigr) - j\log_q(q-1) + j\log_q(en) + 2t\log_q(n) - 2tj .
$$

Since $2tj = nj + j\log_q n + cj$, the terms $nj$ and $j\log_q n$ cancel, so
we are left with

$$
  -\binom{j}{2} - cj + O(n\log n),
$$

where the error term contains $2t\log_q(n) = O(n\log n)$ and
$j\log_q\frac{e}{q-1} = O(n)$. Since $j \ge \frac{4n}{10}$, we have
$\binom{j}{2} \ge \frac{8}{100}n^{2}+O(n)$, and hence each summand is at most
$q^{-\frac{8}{100}n^{2}+O(n\log n)}$. There are at most $n$ possible values of
$j$, so this gives the first term in the bound, which tends to $0$ as
$n \to \infty$.

\emph{The case $\lambda_1 > \frac{6n}{10}$.} In this case
Lemma~\ref{lem:eigbounds}(2) gives
$\bigl|\eig_q(S)\bigr| \le \frac{n+1}{n}q^{-j}$ for every
$S \in \mathrm{SYT}(\lambda)$. We also use
$t_\lambda \le (\frac{q^{n}}{q-1})^{j}$ from Lemma~\ref{lem:bounds}(3), a bound
which depends on $\lambda$ only through $j$, together with
$\sum_{\lambda_1=n-j} d_\lambda \le \frac{(en)^{j}}{\sqrt{j!}}$ from
Lemma~\ref{lem:bounds}(1). It follows that

$$
  \sum_{\substack{\lambda_1 = n-j}}
  t_\lambda\,\chi_\lambda\bigl(P^{2t}\bigr)
  \;\le\;
  \Bigl(\frac{n+1}{n}\Bigr)^{2t}
  \frac{q^{\,nj}}{(q-1)^{j}}\cdot\frac{(en)^{j}}{\sqrt{j!}}\cdot q^{-2tj}.
$$

Since $2t = n+\log_q n + c$, we have $q^{-2tj} = q^{-nj}n^{-j}q^{-cj}$.
Thus $q^{-nj}$ cancels the $q^{nj}$ coming from $t_\lambda$ and $n^{-j}$
cancels the $n^{j}$ coming from the multiplicities, leaving

$$
  \Bigl(\frac{n+1}{n}\Bigr)^{2t}\frac{1}{\sqrt{j!}}
  \left(\frac{e\,q^{-c}}{q-1}\right)^{\!j}
  \;=\;
  \Bigl(\frac{n+1}{n}\Bigr)^{2t}\frac{A^{\,j}}{\sqrt{j!}} .
$$

Finally $\bigl(1+\frac1n\bigr)^{2t}\le e^{\frac{2t}{n}}\le e^{2}$ for $n$ large, and
$j$ ranges over $1 \le j < \frac{4n}{10}$. We can therefore bound the finite
sum by the corresponding infinite series. No restriction on $A$ is needed.
Indeed, by the Cauchy--Schwarz inequality, for any $0<s<1$,

$$
  \sum_{j\ge1}\frac{A^{\,j}}{\sqrt{j!}}
  \;=\; \sum_{j\ge1} s^{\,j}\cdot\frac{(\frac{A}{s})^{j}}{\sqrt{j!}}
  \;\le\; \Bigl(\sum_{j\ge1}s^{2j}\Bigr)^{\frac{1}{2}}
           \Bigl(\sum_{j\ge1}\frac{(\frac{A}{s})^{2j}}{j!}\Bigr)^{\frac{1}{2}}
  \;=\; \Bigl(\frac{s^{2}}{1-s^{2}}\Bigr)^{\frac{1}{2}}
        \bigl(e^{(\frac{A}{s})^2}-1\bigr)^{\frac{1}{2}},
$$

and $s=\frac{1}{\sqrt2}$ gives
$\sum_{j\ge1}\frac{A^{j}}{\sqrt{j!}} \le \bigl(e^{2A^{2}}-1\bigr)^{\frac{1}{2}}$. It is
important that the sum starts at $j=1$: the omitted term $j=0$ equals $1$, and
after removing it the bound tends to $0$ with $A$, hence as $c\to\infty$.
This gives the second term in the bound.

Combining the two cases completes the proof.
\end{proof}

The same argument also controls the $\pi$-average of the $\ell^2$ distances when the
starting state is allowed to vary. Proposition 4.8(b) of \cite{DiaconisRam2000} gives
the corresponding formula, with the generic degree $t_\lambda$ replaced by the
dimension $d_\lambda$. This changes the relevant time scale from
$\frac12\bigl(n+\log_q n\bigr)$ to $\log_q n$. This is natural from the point of
view of the stationary measure: under $\pi$, a typical permutation has Coxeter length
close to $\binom{n}{2}$, so a typical starting state is much closer to equilibrium than
the identity.

\begin{theorem}\label{thm:upperavg}
Let $q>1$ be fixed, let $1 \le k \le n-1$, let $c>0$ and set
$t = \log_q n + c$. Then

$$
  \lim_{n\to\infty}\;
  \sum_{x \in S_n}\pi(x)\Bigl\|\frac{P^{\,t}_{x}}{\pi}-1\Bigr\|_2^{2}
  \;=\;
  \lim_{n\to\infty}\;
  \sum_{\lambda \ne (n)} d_\lambda\,\chi_\lambda\bigl(P^{2t}\bigr)
  \;\le\; e^{\,q^{-2c}}-1 .
$$

\end{theorem}

\begin{proof}
By Proposition 4.8(b) of \cite{DiaconisRam2000}, the left-hand side is
$\sum_{\lambda \ne (n)} d_\lambda \chi_\lambda(P^{2t})$. Also,
$\chi_\lambda(P^{2t}) \le d_\lambda \max_S |\eig_q(S)|^{2t}$ because
$\mathrm{SYT}(\lambda)$ has $d_\lambda$ elements. Write $j = n-\lambda_1$.
We use Lemma~\ref{lem:bounds}(2) in the form

$$\sum_{\lambda_1 = n-j} d_\lambda^{2} \le \frac{n^{2j}}{j!},
$$

and split the sum according to the size of $\lambda_1$. We will repeatedly use
$2t = 2\log_q n + 2c$, which gives
\begin{equation}\label{eq:avgcancel}
q^{-2tj} \;=\; n^{-2j}q^{-2cj}
\qquad\text{and hence}\qquad
\frac{n^{2j}}{j!}q^{-2tj} \;=\; \frac{q^{-2cj}}{j!} .
\end{equation}

\emph{The case $\lambda_1 \le \frac{6n}{10}$.} Here $j \ge \frac{4n}{10}$, and
$|\eig_q(S)| \le \frac{k}{q^{j}}$ by \eqref{eq:z1eig}. Using \eqref{eq:avgcancel}
and $k \le n$, we get

$$
  \sum_{\substack{\lambda \ne (n) \\ \lambda_1 \le \frac{6n}{10}}}
  d_\lambda\,\chi_\lambda\bigl(P^{2t}\bigr)
  \;\le\; \sum_{j \ge \frac{4n}{10}} k^{2t}\,\frac{n^{2j}}{j!}\,q^{-2tj}
  \;=\; k^{2t}\sum_{j \ge \frac{4n}{10}} \frac{q^{-2cj}}{j!}
  \;\le\; \frac{n^{2t}}{\lceil \frac{4n}{10}\rceil!}\,e^{\,q^{-2c}} .
$$

Now $2t = 2\log_q n + 2c$, so $n^{2t} = e^{O((\log n)^{2})}$, whereas
$\lceil \frac{4n}{10}\rceil! = e^{\Theta(n\log n)}$. Therefore this part of the sum tends to $0$.

\emph{The case $\lambda_1 > \frac{6n}{10}$.} Here $1 \le j < \frac{4n}{10}$,
and Lemma~\ref{lem:eigbounds}(2) gives $|\eig_q(S)| \le \frac{n+1}{n}q^{-j}$.
Using \eqref{eq:avgcancel}, we obtain

$$
  \sum_{\substack{\lambda \ne (n) \\ \lambda_1 > \frac{6n}{10}}}
  d_\lambda\,\chi_\lambda\bigl(P^{2t}\bigr)
  \;\le\; \Bigl(\frac{n+1}{n}\Bigr)^{2t}\sum_{j \ge 1}\frac{n^{2j}}{j!}q^{-2tj}
  \;=\; \Bigl(\frac{n+1}{n}\Bigr)^{2t}\sum_{j \ge 1}\frac{q^{-2cj}}{j!}
  \;=\; \Bigl(\frac{n+1}{n}\Bigr)^{2t}\bigl(e^{\,q^{-2c}}-1\bigr).
$$

Finally, $\bigl(1+\frac1n\bigr)^{2t} \le e^{\frac{2t}{n}} \to 1$, since
$t = \log_q n + c = o(n)$. Combining the two cases and letting $n \to \infty$
gives the result.
\end{proof}

The two previous bounds leave one natural question: what happens if the chain starts from a
fixed state other than the identity? For the walks considered here, the identity is
actually the worst starting state. The reason is that changing the starting state inserts
a contraction into the trace formula, and this can only decrease the trace. We first
record the elementary linear algebra fact that we need.

\begin{lemma}\label{lem:trace}
Let $A$ and $B$ be operators on a finite-dimensional inner product space, with
$B$ positive semidefinite and $\|A\| \le 1$, where $\|\cdot\|$ denotes the
operator norm. Then

$$
  \tr\bigl(A^{*}BA\bigr) \;\le\; \tr(B).
$$

\end{lemma}

\begin{proof}
Diagonalize $B$ as
$B = \sum_i \mu_i u_iu_i^{*}$, where the vectors $\mu_i \ge 0$ and ${u_i}$ form an orthonormal basis. By the cyclicity of the trace,
$\tr\bigl(A^{*}u_iu_i^{*}A\bigr) = u_i^{*}AA^{*}u_i = |A^{*}u_i|^{2}$.
Since $|A^{*}u_i| \le \|A\||u_i| \le 1$, it follows that

\[
  \tr\bigl(A^{*}BA\bigr)
  \;=\; \sum_i \mu_i\,\bigl|A^{*}u_i\bigr|^{2}
  \;\le\; \sum_i \mu_i
  \;=\; \tr(B). \qedhere
\]

\end{proof}




\begin{proof}[Proof of Theorem~\ref{T-worst}]
Proposition 4.8(a) of \cite{DiaconisRam2000}, together with
$\ell(x) = \ell(x^{-1})$, gives

$$
  \Bigl\|\frac{X^{\,t}_{x}}{\pi}-1\Bigr\|_2^{2}
  \;=\; \sum_{\lambda \ne \mathbf{1}}
        t_\lambda\,\chi^{\lambda}_{H}
        \bigl(\widetilde{T}_{x}^{*}X^{2t}\widetilde{T}_{x}\bigr),
  \qquad \widetilde{T}_{w} = q^{-\ell(w)}T_{w} .
$$

Now write $x$ as a reduced word. Then $\widetilde{T}^{*}_x$ is a product of the
stochastic matrices $\widetilde{T}^*_{s_i}$ from
Theorem~\ref{thm:metropolis-hecke}. Hence it is itself stochastic with stationary
distribution $\pi$, and Jensen's inequality shows that it is a contraction on
$L^{2}(\pi)$. On the other hand, reversibility of $X$ implies that $X$ is
self-adjoint on $L^{2}(\pi)$, so
$X^{2t} = (X^{t})^{*}X^{t}$ is positive semidefinite. We can therefore apply
Lemma~\ref{lem:trace} and obtain
$\chi^{\lambda}_{H}(\widetilde{T}_{x}^{*}X^{2t}\widetilde{T}_{*}^{x})
\le \chi^{\lambda}_{H}(X^{2t})$ for every $\lambda$, with equality at
$x = \mathrm{id}$. Thus starting from any $x$ can only decrease the
$\ell^2$ distance. Taking $X = P$, the bound from
Theorem~\ref{thm:upper} at $x = \mathrm{id}$ therefore holds for every $x$.
\end{proof}

\subsection{Upper bound for the \texorpdfstring{$q$--random--to--random}{q-random-to-random} shuffle}\label{sec:r2r-upper}

The eigenvalues of $R$ were recorded in \eqref{eq:r2reig}. For the upper bound, we need only the following estimate.

\begin{lemma}\label{lem:r2rbounds}
Let $q \ge 1$, let $\lambda \vdash n$, let
$j = n-\lambda_1$, and let $\mu \subseteq \lambda$ be such that $\lambda/\mu$
is a horizontal strip. Then
$$
  0 \;\le\; \eig_q(\lambda,\mu) \;\le\; \frac{[\lambda_1]_q}{[n]_q}
    \;\le\; \frac{1}{q^{\,j}} ,
$$
where at $q=1$ the $q$-integers are read as $[a]_1 = a$, so that the bound
becomes $0 \le \eig_1(\lambda,\mu) \le \frac{\lambda_1}{n}$.
\end{lemma}

\begin{proof}
Non-negativity holds because $R = \mathcal{B}^{*}_n(q)\,\mathcal{B}_n(q)$.

For the upper bound, by \eqref{eq:r2reig} it suffices to prove $E_{\lambda/\mu}(q) \le [\lambda_1]_q[n]_q$. We bound
the two terms separately and then combine them.

Since $\lambda/\mu$ is a horizontal strip, it contains at most one box in each
column, and $\lambda$ has $\lambda_1$ columns; a box in column $i$ has content
at most $i-1$. As $a \mapsto [a]_q$ is increasing,
\begin{equation}\label{eq:contentcap}
  \sum_{b\in\lambda/\mu}\bigl[c(b)\bigr]_q
  \;\le\; \sum_{i=1}^{\lambda_1}[\,i-1\,]_q .
\end{equation}
The same observation gives $|\lambda/\mu| \le \lambda_1$, that is,
$n - |\mu| \le \lambda_1$. Substituting $r = n+1-i$ in the second term and
extending the resulting range from $i \le n-|\mu|$ to $i \le \lambda_1$, which
only increases the sum since all terms are positive, we obtain
\begin{equation}\label{eq:tailcap}
  \sum_{r=|\mu|+1}^{n} q^{\,n-r}[r]_q
  \;=\; q^{\,n}\sum_{i=1}^{\,n-|\mu|}\frac{[\,n+1-i\,]_q}{q^{\,n+1-i}}
  \;\le\; q^{\,n}\sum_{i=1}^{\lambda_1}\frac{[\,n+1-i\,]_q}{q^{\,n+1-i}} .
\end{equation}
When we add \eqref{eq:contentcap} and \eqref{eq:tailcap}, the two sums run over
the same range, and their $i$-th terms combine by the identity
$\frac{[b]_q + [a]_q}{q^{\,a}} = \frac{[a+b]_q}{q^{\,a}}$, applied with $b = i-1$ and
$a = n+1-i$, so that $a+b = n$ for every $i$:
$$
  E_{\lambda/\mu}(q)
  \;\le\; q^{\,n}\sum_{i=1}^{\lambda_1}
          \left([\,i-1\,]_q + \frac{[\,n+1-i\,]_q}{q^{\,n+1-i}}\right)
  \;=\; q^{\,n}\sum_{i=1}^{\lambda_1}\frac{[n]_q}{q^{\,n+1-i}}
  \;=\; [n]_q\sum_{i=1}^{\lambda_1} q^{\,i-1}
  \;=\; [n]_q\,[\lambda_1]_q .
$$
This proves $\eig_q(\lambda,\mu) \le \frac{[\lambda_1]_q}{[n]_q}$. The final
inequality is $[\lambda_1]_q \le q^{\lambda_1-n}[n]_q$, which after clearing
denominators reduces to $q^{\lambda_1} \le q^{\,n}$. Every step above uses
only that $a \mapsto [a]_q$ is increasing and that $q$-integers are positive,
both of which hold for all $q \ge 1$; at $q = 1$ the argument specializes to
the corresponding statement for the random--to--random shuffle of
\cite{Dieker2018}.
\end{proof}

  In fact, the eigenvalue bound for $R$ is stronger than either of the bounds used for $P$ in
Section~\ref{sec:kstar-upper}. Therefore the same argument applies to $R$ without any
additional estimates.

\begin{corollary}\label{cor:r2r}
Theorems~\ref{thm:upper} and~\ref{thm:upperavg}  hold with $P$ replaced by $R$.
In particular, by Theorem~\ref{T-worst}, $\frac12\bigl(n+\log_q n+c\bigr)$ steps suffice for the
$q$--random--to--random shuffle to mix from any starting permutation.
\end{corollary}

\begin{proof}
In the proofs above, the eigenvalues of $P$ are used only through the bounds
$|\eig_q(S)| \le kq^{-j}$ and $|\eig_q(S)| \le \frac{n+1}{n}q^{-j}$.
By Lemma~\ref{lem:r2rbounds}, the eigenvalues of $R$ satisfy
$\eig_q(\lambda,\mu) \le q^{-j}$, so the same estimates apply, with a stronger
eigenvalue bound. Moreover, $R$ is self-adjoint since it is of the form
$a^{*}a$. All the other parts of the proofs are independent of the choice of
the chain, so the results follow for $R$ as well.
\end{proof}

\subsection{The short systematic scan}\label{sec:scan-upper}

Let $K = \widetilde{T}_{s_{n-1}}\cdots\widetilde{T}_{s_{2}}\widetilde{T}_{s_{1}}^{2}
\widetilde{T}_{s_{2}}\cdots\widetilde{T}_{s_{n-1}}$ be the short systematic
scan Metropolis chain on $S_n$ of \cite{DiaconisRam2000}. We improve the
coefficient of the logarithmic term in the upper bound of \cite{DiaconisRam2000}
from $1$ to $\frac12$.

\begin{lemma}\label{lem:scan}
Let $\lambda \vdash n$ with $\lambda \ne (n)$ and $j = n-\lambda_1$. Every
eigenvalue of $K$ on $S^{\lambda}_q$ satisfies
$$
  0 \;<\; \eig_q(S) \;\le\; \frac{1}{q^{\,j}} .
$$
\end{lemma}

\begin{proof}
By \eqref{eq:scaneig}, the eigenvalues of $K$ on $S^{\lambda}_q$ are given by
$\eig_q(S) = q^{\,c_n(S)-n+1}$, which is positive. Every box of $\lambda$ has content at
most $\lambda_1-1$, so $c_n(S) \le \lambda_1-1$ and hence
$\eig_q(S) \le q^{\lambda_1-n} = q^{-j}$.
\end{proof}


\begin{proposition}\label{prop:k1scan}
The short systematic scan and the $q$--deformed $k$--star transposition shuffle at $k=1$ are both affine functions
of the single Jucys--Murphy element $J_n(q)$:
\[
  P\big|_{k=1} \;=\; \frac{I + q\,J_n(q)}{[n]_q},
  \qquad
  K \;=\; q^{\,1-n}\bigl(I + (q-1)\,J_n(q)\bigr).
\]
For $\lambda\vdash n$ and $S\in\mathrm{SYT}(\lambda)$, write $c=c_n(S)$.
The corresponding eigenvalues satisfy
\[
  \eig_q^{\,P|_{k=1}}(S)=\frac{[c+1]_q}{[n]_q},
  \qquad
  \eig_q^{\,K}(S)=q^{c+1-n}.
\]
\end{proposition}

\begin{proof}
For $k=1$, the definition of $P$ gives
\[
  P\big|_{k=1}
  = \frac{1}{[n]_q}I
    + \frac{q[n-1]_q}{[n]_q[n-1]_q}J_n(q)
  = \frac{I+qJ_n(q)}{[n]_q}.
\]
For the short systematic scan, recall that $\theta=q^{-1}$ and
\[
  K=\widetilde T_{n-1}\cdots\widetilde T_2
    \widetilde T_1^{\,2}
    \widetilde T_2\cdots\widetilde T_{n-1}.
\]
Applying the quadratic relation
\[
  \widetilde T_i^{\,2}=(1-\theta)\widetilde T_i+\theta I
\]
successively, starting with the central factor $\widetilde T_1^{\,2}$,
we obtain
\[
\begin{aligned}
  K
  &= \theta^{n-1}I
     +(1-\theta)\sum_{i=1}^{n-1}
       \theta^{i-1}\widetilde T_{(i\,n)} \\
  &= q^{1-n}\left(
       I+(q-1)\sum_{i=1}^{n-1}
       q^{n-1-i}\widetilde T_{(i\,n)}
     \right) \\
  &= q^{1-n}\bigl(I+(q-1)J_n(q)\bigr),
\end{aligned}
\]
where the last equality is the definition of $J_n(q)$.

Thus $K$ and $P|_{k=1}$ are affine functions of $J_n(q)$, so they
commute and share its eigenvectors. Substituting the eigenvalue $[c]_q$
of $J_n(q)$ into these expressions and using
\[
  1+q[c]_q=[c+1]_q,
  \qquad
  1+(q-1)[c]_q=q^c,
\]
gives the stated eigenvalues.
\end{proof}

With Lemma~\ref{lem:scan} in hand the upper bound for $K$ needs no new work.

\begin{theorem}\label{thm:scan}
Let $q>1$ be fixed and let $c>0$. Started at the identity, and with
$$
  t \;=\; \tfrac12\bigl(n+\log_q n + c\bigr),
  \qquad
  A \;=\; \frac{e\,q^{-c}}{q-1},
$$
the short systematic scan satisfies, for every $n$,
$$
  \Bigl\|\frac{K^{\,t}_{\mathrm{id}}}{\pi}-1\Bigr\|_{2}^{2}
  \;\le\; \sum_{j\ge1}\frac{A^{\,j}}{\sqrt{j!}}
  \;\le\; \sqrt{e^{2A^{2}}-1}\, .
$$
In particular, $\frac12\bigl(n+\log_q n+c\bigr)$ scans suffice to mix in
$\ell^2$ and in total variation.
\end{theorem}

\begin{proof}
By Lemma~\ref{lem:scan} the eigenvalues of $K$ satisfy
$|\eig_q(S)| \le q^{-j}$, which is stronger than both bounds used in the proof
of Theorem~\ref{thm:upper}, namely $k\,q^{-j}$ and $\frac{n+1}{n}q^{-j}$; moreover $K$ is
self-adjoint, being of the form $a^{*}a$ with
$a = \widetilde{T}_{s_{1}}\widetilde{T}_{s_{2}}\cdots\widetilde{T}_{s_{n-1}}$.
Since no other property of the chain enters, the proof of
Theorem~\ref{thm:upper} applies verbatim, and as in
Section~\ref{sec:r2r-upper} no splitting according to the size of
$\lambda_1$ is required.
\end{proof}


\begin{remark}\label{rem:gap}
All three shuffles share one feature: on every $S^{\lambda}_q$ with $\lambda \ne (n)$ their
eigenvalues are bounded away from $1$. For the scan and for $R$ this is exact,
$$
  0 \;\le\; \eig_q(S) \;\le\; \frac{1}{q^{\,j}} \;\le\; \frac1q ,
  \qquad j = n-\lambda_1 \ge 1 ,
$$
by Lemmas~\ref{lem:scan} and~\ref{lem:r2rbounds}; for $P$ the same holds with $\frac1q$
replaced by $\frac{n+1}{n}\cdot\frac1q$, by Lemma~\ref{lem:eigbounds}(2) and
\eqref{eq:z1eig}. So for fixed $q>1$ the spectral gap does not close as $n$ grows.

This is what the deformation buys. At $q=1$ the largest eigenvalue on $S^{(n-1,1)}$ is
$1-\Theta(\frac{1}{n})$, so of order $n\log n$ steps are needed before the dimensions $d_\lambda$
are overcome. At $q>1$ the obstruction is not the eigenvalues but the weights: by
Lemma~\ref{lem:bounds}(3) the generic degree satisfies
$t_\lambda \le \bigl(\frac{q^{\,n}}{q-1}\bigr)^{j}$, the multiplicities contribute a further
$n^{j}$, and $q^{-2tj}$ kills both once $2t$ passes $n$ and $\log_q n$, respectively. That
is where the $\frac n2$ and the $\frac12\log_q n$ of Theorem~\ref{T-l2cutoff} come from.
\end{remark}




\section{\texorpdfstring{$\ell^{2}$}{l2} lower bound}\label{sec:l2lower}

We show that the time $\frac12\bigl(n+\log_q n\bigr)$ of
Theorem~\ref{thm:upper} cannot be improved.

In each case, a single partition suffices, namely $\lambda = (n-1,1)$. Its
generic degree is large, its $d_\lambda = n-1$ eigenvalues are all close to
$q^{-1}$, and this alone forces the $\ell^{2}$ distance to blow up.

Throughout this section, we fix $c>0$ and set
\begin{equation}\label{eq:t-def}
  t \;=\; \tfrac12\bigl(n + \log_q n - c\bigr),
  \qquad\text{so that}\qquad
  q^{-2t} \;=\; \frac{q^{\,c}}{n\,q^{\,n}} .
\end{equation}

Since $\lambda = (n-1,1)$ has $n(\lambda)=1$ and hook lengths
$n,\,1,\,2,\dots,n-2$ together with one additional box of hook length $1$,
\begin{equation}\label{eq:tn11}
  t_{(n-1,1)}
  \;=\; q\,[n-1]_q
  \;=\; \frac{q^{\,n}-q}{q-1}
  \;>\; q^{\,n-1}
  \qquad (n \ge 3).
\end{equation}

Combining \eqref{eq:t-def} and \eqref{eq:tn11}, we obtain directly
\begin{equation}\label{eq:master-lb}
  t_{(n-1,1)}\cdot \frac{n}{2}\cdot q^{-2t}
  \;>\; q^{\,n-1}\cdot\frac{n}{2}\cdot\frac{q^{\,c}}{n\,q^{\,n}}
  \;=\; \frac{q^{\,c-1}}{2} .
\end{equation}

Each of the three subsections below identifies at least $\frac n2$ eigenvalues
of the relevant chain on $S^{(n-1,1)}_q$ that are at least
$q^{-1}\bigl(1-o(\frac{1}{n})\bigr)$, and then applies \eqref{eq:master-lb}.

\bigskip

\subsection{The \texorpdfstring{$q$--deformed $k$--star}{q-deformed k-star} transposition shuffle}

\begin{lemma}\label{lem:kstar-eigs}
Let $\Sigma_k = \sum_{l=1}^{k}[n-l]_q$. On $S^{(n-1,1)}_q$ the operator $P$ has
exactly two eigenvalues,
$$
  \beta_1 \;=\; 1 - \frac{[n-1]_q}{\Sigma_k},
  \qquad\text{of multiplicity } k,
$$
$$
  \beta_2 \;=\; 1 - \frac{[n-1]_q}{[n]_q}\cdot
                    \frac{q^{\,n-k}[k]_q}{\Sigma_k},
  \qquad\text{of multiplicity } n-1-k .
$$
\end{lemma}

\begin{figure}[ht]
\centering
\includegraphics[width=420pt]{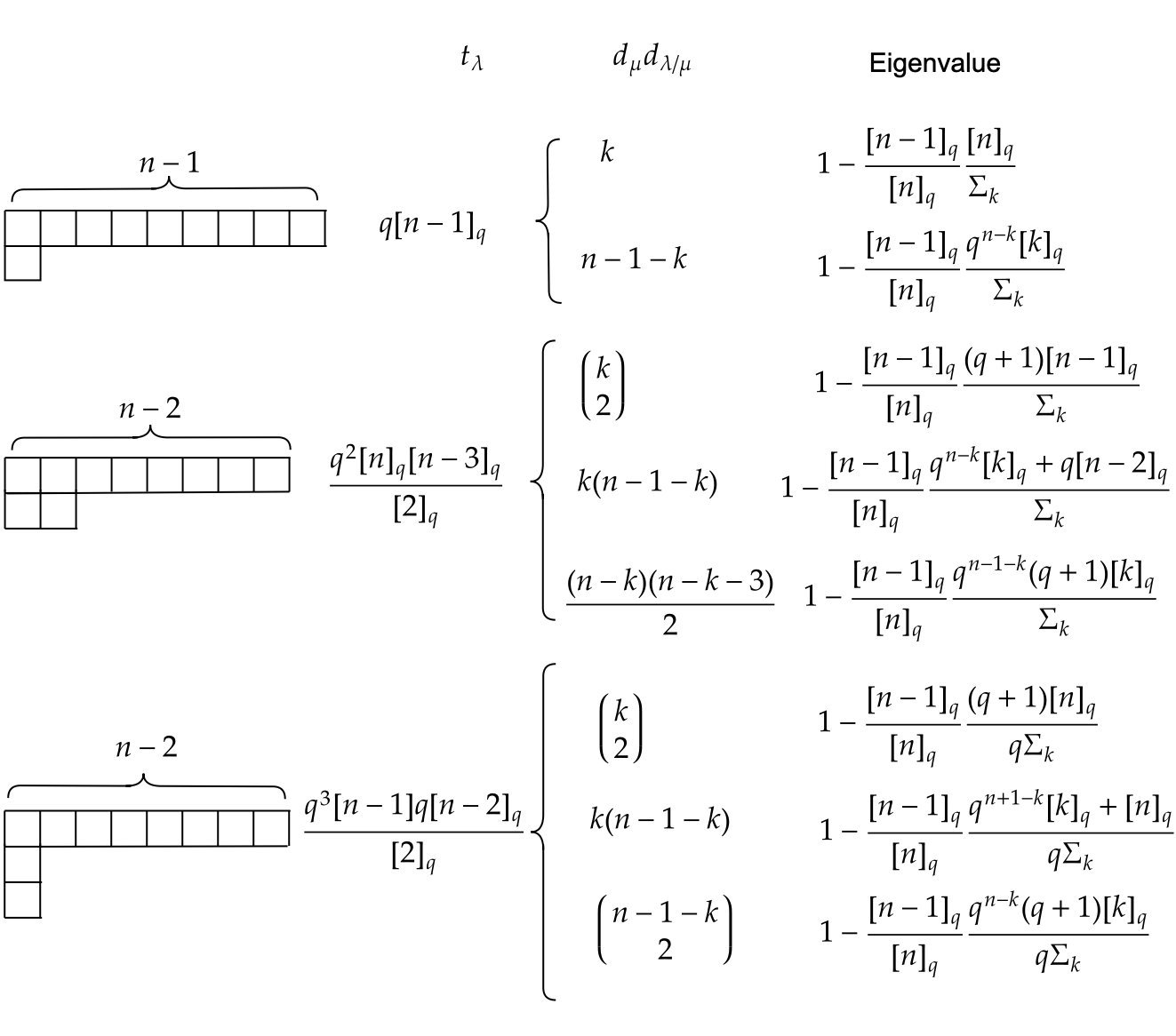}
\caption{Generic degree, multiplicity and eigenvalue of $P$ for the three partitions with
$\lambda_1 \ge n-2$, in the notation of Theorem~\ref{T-EVHeckeCombined}. The first row is
the case $\lambda = (n-1,1)$ of Lemma~\ref{lem:kstar-eigs}.}
\label{fig:eigtable}
\end{figure}

Figure~\ref{fig:eigtable} records these two eigenvalues together with the corresponding
data for the two partitions with $\lambda_1 = n-2$.

\begin{lemma}\label{lem:kstar-lb-eig}
For $n$ sufficiently large,
$$
  \beta_2 \;\ge\; \frac1q - \frac{k\,(q-1)}{q^{\,n-1}-1},
  \qquad\text{and, if } k > \tfrac n2, \qquad
  \beta_1 \;\ge\; \frac1q - \frac{q^{\,n-k}+kq}{q^{\,n}} .
$$
In particular, in either case the relevant eigenvalue is at least
$\frac1q\bigl(1-2q^{\,1-\frac{n}{2}}\bigr)$.
\end{lemma}

\begin{proof}
Summing a geometric series gives $\sum_{l=1}^{k}q^{\,n-l} = q^{\,n-k}[k]_q$,
and therefore
\begin{equation}\label{eq:sigmaid}
  q^{\,n-k}[k]_q \;=\; (q-1)\,\Sigma_k + k .
\end{equation}

Substituting \eqref{eq:sigmaid} into the expression for $\beta_2$ gives
$$
  \beta_2
  \;=\; 1 - \frac{[n-1]_q}{[n]_q}\Bigl((q-1) + \frac{k}{\Sigma_k}\Bigr)
  \;=\; \underbrace{\frac{q^{\,n-1}+q-2}{q^{\,n}-1}}_{\textstyle \ge\ q^{-1}}
        \;-\; \frac{[n-1]_q}{[n]_q}\cdot\frac{k}{\Sigma_k} ,
$$
the indicated inequality being equivalent to $(q-1)^{2} \ge 0$. Since
$\frac{[n-1]_q}{[n]_q} \le 1$ and $\Sigma_k \ge [n-1]_q$, the subtracted term
is at most $\frac{k(q-1)}{q^{\,n-1}-1}$, which gives the first bound.

The second bound follows in the same way from
$$\Sigma_k = \frac{\bigl(q^{\,n}-q^{\,n-k}-k(q-1)\bigr)}{\big(q-1\big)^{2}}.$$

Finally, if $k \le \frac n2$ the first correction is at most
$\frac{n(q-1)}{q^{\,n-1}-1}$, and if $k > \frac n2$ the second is at most
$q^{-\frac{n}{2}}+kq^{\,1-n}$; both are below $\frac1q\cdot 2q^{\,1-\frac{n}{2}}$ for $n$
large.
\end{proof}

Both eigenvalues are therefore within $O\bigl(q^{\,1-\frac{n}{2}}\bigr)$ of $q^{-1}$, which is
what \eqref{eq:master-lb} needs.

\begin{theorem}\label{thm:kstar-lb}
Let $q>1$ be fixed and $1 \le k \le n-1$. Then, for $n$ sufficiently large,
$$
  \Bigl\|\frac{P^{\,t}_{\mathrm{id}}}{\pi}-1\Bigr\|_2^{2}
  \;\ge\; \frac{q^{\,c-1}}{4} .
$$
\end{theorem}

\begin{proof}
All terms of \eqref{eq:master} are non-negative, so we may keep only the
partition $\lambda = (n-1,1)$.

If $k \le \frac n2$, we retain only the eigenvalue $\beta_2$, whose
multiplicity is $n-1-k \ge \frac n2 - 1$; if $k > \frac n2$ we retain only
$\beta_1$, whose multiplicity is $k > \frac n2$. By
Lemma~\ref{lem:kstar-lb-eig}, in either case
$$
  \sum_{S \in \mathrm{SYT}(n-1,1)} \eig_q(S)^{2t}
  \;\ge\; \Bigl(\frac n2-1\Bigr)\,q^{-2t}\,
          \Bigl(1-2q^{\,1-\frac{n}{2}}\Bigr)^{2t} .
$$

Since $2t = n+\log_q n - c \le 2n$ for $n$ large, Bernoulli's inequality bounds the last
factor below by $1-4t\,q^{\,1-\frac{n}{2}} \ge 1-4n\,q^{\,1-\frac{n}{2}}$, which tends to $1$. Combining this bound with \eqref{eq:master-lb} gives
$\frac{q^{\,c-1}}{2}\bigl(1+o(1)\bigr)$, which exceeds $\frac{q^{\,c-1}}{4}$
for $n$ large.
\end{proof}

\bigskip

\subsection{The \texorpdfstring{$q$--random--to--random}{q-random-to-random} shuffle}

\begin{lemma}\label{lem:r2r-eigs}
The partitions $\mu \subseteq (n-1,1)$ for which $(n-1,1)/\mu$ is a horizontal
strip are exactly $\mu = (j,1)$ with $1 \le j \le n-1$, each with
$d^{\mu} = 1$, and
$$
  \eig_q\bigl((n-1,1),(j,1)\bigr)
  \;=\; \frac{[\,n-1-j\,]_q\,[\,n+j\,]_q}{[n]_q^{2}} .
$$
Moreover
$$
  \eig_q\bigl((n-1,1),(j,1)\bigr) \;\ge\; \frac1q\Bigl(1-q^{\,j-n+1}\Bigr)
  \qquad\text{for all } 1 \le j \le n-1 .
$$
\end{lemma}

\begin{proof}
After we expand the $q$-integers and clear the denominators, the asserted inequality
$$
  \frac{\bigl(q^{\,n-1-j}-1\bigr)\bigl(q^{\,n+j}-1\bigr)}
       {\bigl(q^{\,n}-1\bigr)^{2}}
  \;\ge\; \frac1q\bigl(1-q^{\,j-n+1}\bigr)
$$
reduces to $q^{\,n-1}\bigl(2-q^{-j}\bigr) \ge q^{\,j}+q^{-1}-q^{\,j-n}$, which
holds for every $1 \le j \le n-1$. No error term is required: the omitted
terms are positive.
\end{proof}

Restricting to $j \le \frac n2$ makes the error term $q^{\,j-n+1}$ negligible, and
\eqref{eq:master-lb} applies.

\begin{theorem}\label{thm:r2r-lb}
Let $q>1$ be fixed. Then, for $n$ sufficiently large,
$$
  \Bigl\|\frac{R^{\,t}_{\mathrm{id}}}{\pi}-1\Bigr\|_2^{2}
  \;\ge\; \frac{q^{\,c-1}}{4} .
$$
\end{theorem}

\begin{proof}
Keep only $\lambda=(n-1,1)$ and the indices
$1\leq j\leq \frac{n}{2}$. By Lemma~\ref{lem:r2r-eigs}, these give at least $\frac{n}{2}$
eigenvalues bounded below by
\[
\frac1q\left(1-q^{1-\frac{n}{2}}\right).
\]
The same argument as in the proof of Theorem~\ref{thm:kstar-lb}, together with
\eqref{eq:master-lb}, gives
\[
\left\|
\frac{R_{\mathrm{id}}^t}{\pi}-1
\right\|_2^2
\geq
\frac{q^{c-1}}{4}
\]
for all sufficiently large $n$.
\end{proof}

\bigskip

\subsection{The short systematic scan}

\begin{lemma}\label{lem:scan-eigs}
Write $S_b$ for the standard tableau of shape $(n-1,1)$ with $S(2,1) = b$,
where $2 \le b \le n$. Then
$$
  \eig_q(S_b) \;=\; \frac1q \quad (b \ne n),
  \qquad\qquad
  \eig_q(S_n) \;=\; \frac{1}{q^{\,n}} .
$$
In particular, exactly $n-2$ of the $n-1$ eigenvalues equal $q^{-1}$.
\end{lemma}

\begin{proof}
By Theorem 7.5(a) of \cite{DiaconisRam2000}, the eigenvalue attached to a
standard tableau $S$ is $$\theta^{\,n-1-c(S(n))} = q^{\,c(S(n))-n+1},$$ where
$S(n)$ is the box containing $n$.

If $b \ne n$, the entry $n$ lies at the end of the first row, in column $n-1$,
so $c(S_b(n)) = n-2$ and the eigenvalue is $q^{-1}$. If $b = n$, the entry $n$
lies in the box $(2,1)$, of content $-1$, and the eigenvalue is $q^{-n}$.
\end{proof}

Here the eigenvalue $q^{-1}$ is exact, so no error term arises at all.
\begin{theorem}\label{thm:scan-lb}
Let $q>1$ be fixed. Then, for $n \ge 4$,
$$
  \Bigl\|\frac{K^{\,t}_{\mathrm{id}}}{\pi}-1\Bigr\|_2^{2}
  \;\ge\; \frac{q^{\,c-1}}{2} .
$$
\end{theorem}

\begin{proof}
By Lemma~\ref{lem:scan-eigs},
$$
  \sum_{b=2}^{n}\eig_q(S_b)^{2t}
  \;\ge\; (n-2)\,q^{-2t}
  \;\ge\; \frac n2\,q^{-2t}
  \qquad (n \ge 4),
$$
and \eqref{eq:master-lb} gives the claim. No error terms arise, the
eigenvalue $q^{-1}$ being exact.
\end{proof}

The three lower bounds above meet the upper bounds of
Section~\ref{sec:boundeigen} at the same time, which gives us the cutoff.

\begin{proof}[Proof of Theorem~\ref{T-l2cutoff}]
For the upper bound, Theorem~\ref{thm:upper}, Corollary~\ref{cor:r2r}, and Theorem~\ref{thm:scan} show, for $P$, $R$, and $K$ respectively, that \[ t_+=\frac12\bigl(n+\log_q n+c\bigr) \] steps are sufficient for the $\ell^2$ distance from the identity to tend to zero as first $n\to\infty$ and then $c\to\infty$. For the lower bound, Theorems~\ref{thm:kstar-lb}, \ref{thm:r2r-lb} and \ref{thm:scan-lb} show that at \[ t_-=\frac12\bigl(n+\log_q n-c\bigr), \] for all sufficiently large $n$, \[ \left\| \frac{X_{\mathrm{id}}^{t_-}}{\pi}-1 \right\|_2^2 \ge \frac{q^{c-1}}{4} \] for each of the three chains. Therefore the $q$--deformed $k$--star transposition shuffle, the $q$--random--to--random shuffle, and the short systematic scan all exhibit $\ell^2$ cutoff at \[ \frac12\bigl(n+\log_q n\bigr) \] with a window of order $1$. \end{proof}







\section{The three shuffles at \texorpdfstring{$\theta = 0$}{theta}}\label{sec:theta0}

We now show that at $\theta = 0$ the three shuffles all become the short systematic scan.

Substituting $J_m(q) = \sum_{i=1}^{m-1}q^{\,m-1-i}\widetilde T_{(i\,m)}$ and
$m = n+1-l$ into \eqref{eq:Pdef}, we can write the $q$--deformed $k$--star transposition shuffle as
\begin{equation}\label{eq:Pbasis}
  P \;=\; \frac{1}{[n]_q}\,\widetilde T_{e}
     \;+\; \frac{q\,[n-1]_q}{[n]_q\,\Sigma_k}
       \sum_{l=1}^{k}\;\sum_{i=1}^{n-l} q^{\,n-l-i}\;
       \widetilde T_{(i,\,n+1-l)},
  \qquad
  \Sigma_k \;=\; \sum_{l=1}^{k}[n-l]_q ,
\end{equation}
a convex combination of basis elements: the coefficients sum to
$\frac{\bigl(1+q[n-1]_q\bigr)}{[n]_q} = 1$, using $\sum_{i=1}^{m-1}q^{\,m-1-i} = [m-1]_q$ and
$\sum_{l=1}^{k}[n-l]_q = \Sigma_k$.



\begin{theorem}\label{prop:theta0}
Fix $n\ge 2$ and $1\le k\le n-1$. At $\theta=0$, the transition operators
of the $q$--deformed $k$--star transposition shuffle $P$, the
$q$--random--to--random shuffle $R$, and the short systematic scan $K$
satisfy
\[
  \left.P\right|_{\theta=0}
  =\left.R\right|_{\theta=0}
  =\left.K\right|_{\theta=0}
  =\widetilde T_{(1\,n)}
  \qquad\text{in the $0$-Hecke algebra}.
\]
Thus one step of each shuffle sends $w\in S_n$ to $(1\,n)\star w$,
where $\star$ denotes the Demazure product of Definition~\ref{def:zerohecke}.
\end{theorem}

\begin{proof}
We use the following fact twice. Both
$$
  s_1 s_2\cdots s_{n-1}s_{n-2}\cdots s_1
  \qquad\text{and}\qquad
  s_{n-1}\cdots s_2 s_1 s_2\cdots s_{n-1}
$$
are reduced words for $(1\,n)$, of length $2n-3 = \ell\bigl((1\,n)\bigr)$.

\emph{The scan.} By Definition~\ref{def:three}(1), $K|_{\theta=0}$ is
$\widetilde T_{v}$ with $v$ the greedy product of
$$s_{n-1},\dots,s_2,s_1,s_1,s_2,\dots,s_{n-1}.$$ The repeated letter is absorbed by
$\widetilde T_{s_1}^{2} = \widetilde T_{s_1}$, and the remaining word
$s_{n-1}\cdots s_2s_1s_2\cdots s_{n-1}$ is reduced, so its greedy product is its ordinary
product $(1\,n)$. Hence $K|_{\theta=0} = \widetilde T_{(1\,n)}$.

\emph{The $k$--star shuffle.} By \eqref{eq:Pbasis} the coefficient of
$\widetilde T_{(i,\,n+1-l)}$ in $P$ is
$$
  \frac{q\,[n-1]_q}{[n]_q\,\Sigma_k}\;q^{\,n-l-i} .
$$
By Lemma~\ref{lem:qinteger-facts}(\ref{it:qintlead}), $[n]_q \sim q^{\,n-1}$, $[n-1]_q \sim q^{\,n-2}$ and
$\Sigma_k \sim [n-1]_q \sim q^{\,n-2}$, so this coefficient behaves like
$$
  q \cdot q^{\,n-2} \cdot q^{-(n-1)} \cdot q^{-(n-2)} \cdot q^{\,n-l-i}
  \;=\; q^{\,2-l-i},
$$
which tends to $0$ for every pair $(l,i) \neq (1,1)$ with $l,i \ge 1$, and to $1$ for
$(l,i) = (1,1)$; the coefficient $\frac{1}{[n]_q}$ of $\widetilde T_{e}$ also tends to $0$. Since
$(l,i) = (1,1)$ corresponds to $(i,\,n+1-l) = (1\,n)$, we get
$P|_{\theta=0} = \widetilde T_{(1\,n)}$, for every $k$.

\emph{Random--to--random.} By Lemma~\ref{lem:qinteger-facts}\ref{it:qintlead} the weights $\frac{q^{\,n-i}}{[n]_q}$ and
$\frac{q^{\,n-j}}{[n]_q}$ in Definition~\ref{def:three}(2) tend to $\delta_{i,1}$ and
$\delta_{j,1}$, so
$R|_{\theta=0} = \widetilde T_{s_1\cdots s_{n-1}}\,\widetilde T_{s_{n-1}\cdots s_1}
= \widetilde T_{v}$ with $v$ the greedy product of
$s_1\cdots s_{n-1}s_{n-1}\cdots s_1$. Again the repeated $s_{n-1}$ is absorbed, and the
remaining word $$s_1\cdots s_{n-1}s_{n-2}\cdots s_1$$ is reduced with product $(1\,n)$. Hence
$R|_{\theta=0} = \widetilde T_{(1\,n)}$.
\end{proof}

We now describe the orbit of the identity under this map.

\begin{corollary}\label{cor:theta0orbit}
Let $w^{(t)}=K_0^t(\mathrm{id})$, where $K_0$ is the short systematic
scan at $\theta=0$. For $0\le t\le\lfloor\frac{n}{2}\rfloor$,
\[
  w^{(t)}=\prod_{i=1}^{t}(i\,\,n+1-i),\qquad
  D(w^{(t)})=\binom{n-2t}{2},\qquad
  \ell(w^{(t)})=2tn-2t^2-t.
\]
The scan first reaches the longest permutation $w_0$ at
$t=\lfloor\frac{n}{2}\rfloor$.
\end{corollary}

\begin{proof}
Write the inverse permutation in one-line notation. At $\theta=0$, we
may scan from left to right and then back, swapping adjacent entries
whenever the smaller one is on the left. Starting from the identity,
the first sweep moves $1$ all the way to the right. On the way back,
$n$ takes the first position, while $2,\ldots,n-1$ stay in increasing
order.

With $n$ first and $1$ last, neither can move again. The next scan
therefore acts on $2,\ldots,n-1$ just as the first scan acted on
$1,\ldots,n$. Repeating this argument shows that the $t$th scan
exchanges $t$ and $n+1-t$, leaving the entries between them in their
original order. This gives the stated product for $w^{(t)}$.

After $t$ scans, the only pairs in increasing order are those within
the middle block $t+1,\ldots,n-t$. There are $\binom{n-2t}{2}$ such
pairs, so $\ell(w^{(t)})=\binom{n}{2}-\binom{n-2t}{2}=2tn-2t^2-t$.
The permutation is decreasing once the middle block has at most one
entry. This first happens at $t=\lfloor\frac{n}{2}\rfloor$.
\end{proof}


\section{Total variation lower bound}\label{sec:l1lower}

By \eqref{eq:tvl2} the $\ell^{2}$ upper bounds of Section~\ref{sec:boundeigen} are also total
variation upper bounds, but the $\ell^{2}$ lower bounds of Section~\ref{sec:l2lower} say
nothing about total variation. We close with a lower bound in total variation itself.

Write $K_0$ for the short systematic scan run at $\theta = 0$, that is, for the
deterministic map of Corollary~\ref{cor:theta0orbit}.

\begin{lemma}\label{lem:coupling}
Let $T_{w_1},\ldots,T_{w_t} \in H(q)$. Then the element of maximal Coxeter length appearing in the expansion of $T_{w_1}\ldots T_{w_t}$ is exactly $w_1 \star \ldots \star w_t$. Consequently, if $(X^t)$ is a Markov chain and $(Z^t_X)$ is its $\theta=0$ chain, as defined in Section~\ref{sec:demazure}, then

$$
\ell(X^t) \le \ell(Z^t_X)
$$

for all $t$.
\end{lemma}

\begin{proof}
The multiplication of two words $\tilde T_{w} \tilde T_{w'}$ in $H(q)$ is equal to a linear combination $\sum_{v \in W} c_v\tilde T_v$. It is a basic fact that for any $\sigma_1,\ldots,\sigma_t \in {s_1,\ldots,s_{n-1}}$, $\sigma_1\ldots \sigma_t \le \sigma_1 \star \ldots \star \sigma_t$. The lemma follows from this fact and the construction of the chain $(Z^t_X)$.
\end{proof}

\begin{corollary}\label{cor:reduce}
Let $(X^{t})$ be one of the three chains: the short systematic scan $K$, the
$q$--deformed $k$--star transposition shuffle $P$ with
$1\leq k\leq n-1$, or the $q$--random--to--random shuffle $R$. Suppose the
chain is started at the identity. Then, for every 
$0\leq t\leq\lfloor \frac{n}{2}\rfloor$,

$$
  \ell\bigl(X^{t}\bigr)
  \;\leq\;
  \ell\bigl(K_{0}^{t}(\mathrm{id})\bigr)
  \;=\;
  2tn-2t^{2}-t.$$  
  \end{corollary}
\
\begin{proof}

By Theorem~\ref{prop:theta0}, the $\theta=0$ chain associated with each of
$P$, $R$, and $K$ is the short systematic scan $K_0$.
Lemma~\ref{lem:coupling} therefore gives
\[
\ell(X^t)\le \ell(Z^t_X) \le \ell\bigl(K_0^t(\mathrm{id})\bigr) .
\]
The value on the right is computed in Corollary~\ref{cor:theta0orbit}, giving
\[
\ell(X^t)\le 2tn-2t^2-t.
\]
\end{proof}

\begin{lemma}[{\cite[Equation 2.13]{DiaconisRam2000}}]\label{lem:statistics}
In the type $A$ Hecke algebra $H(\theta)$ associated to $S_n$,
$$
  \mathbb{E}_\pi(\ell) = \sum_{j=2}^{n}\frac{j}{1-\theta^{j}} - \frac{n-1}{1-\theta},
  \qquad
  \mathrm{Var}_\pi(\ell) = \frac{(n-1)\theta}{(1-\theta)^{2}}
      - \sum_{j=2}^{n}\frac{j^{2}\theta^{j}}{(1-\theta^{j})^{2}} .
$$
In particular, if we write $D(w) = \binom n2 - \ell(w)$ and
$\gamma = \dfrac{\theta}{1-\theta} = \dfrac{1}{q-1}$, the first identity rearranges to
\begin{equation}\label{eq:EDexact}
  \mathbb{E}_\pi(D)
  \;=\; \gamma\,(n-1) \;-\; \sum_{j=2}^{n}\frac{j\,\theta^{j}}{1-\theta^{j}} ,
\end{equation}
so that $\mathbb{E}_\pi(D) \le \gamma\,(n-1)$, and
$$
  \mathbb{E}_\pi(D) = \gamma\,(n-1) + O(1),
  \qquad
  \mathrm{Var}_\pi(D) = \mathrm{Var}_\pi(\ell)
    = \frac{\gamma\,(n-1)}{1-\theta} + O(1),
$$
both of order $n$ for fixed $\theta \in (0,1)$.
\end{lemma}

\begin{proof}
The two displayed identities follow from \cite[Equation 2.13]{DiaconisRam2000}. For
\eqref{eq:EDexact}, write $\frac{j}{1-\theta^{j}} = j + \frac{j\theta^{j}}{1-\theta^{j}}$
and use $\sum_{j=2}^{n} j = \binom n2 + n - 1$, so that
$$
  \mathbb{E}_\pi(D)
  = \binom n2 - \mathbb{E}_\pi(\ell)
  = \frac{n-1}{1-\theta} - (n-1) - \sum_{j=2}^{n}\frac{j\theta^{j}}{1-\theta^{j}} ,
$$
and $\frac{1}{1-\theta}-1 = \gamma$. Each term of the remaining sum is positive and the
whole sum is $O(1)$ for fixed $\theta<1$, giving both the inequality and the estimate. The
same expansion gives $\mathrm{Var}_\pi(\ell) = \frac{(n-1)\theta}{(1-\theta)^{2}}+O(1)$ and
$\frac{\theta}{(1-\theta)^{2}} = \frac{\gamma}{1-\theta}$.
\end{proof}

Corollary~\ref{cor:reduce} and Lemma~\ref{lem:statistics} are the two halves of the
argument: the first says the chain cannot yet have sorted the deck, and the second says that a
$\pi$--typical permutation is sorted. Comparing them separates the two laws.

\begin{theorem}\label{thm:l1}
Let $q>1$ be fixed and let $(X^{t})$ be the chain determined by multiplication by $P$, for
any $1\le k\le n-1$, by $R$ or by $K$. Then for every integer $t$ with $n-2t\ge2$ and
$\binom{n-2t}{2} > \mathbb{E}_\pi(D)$,
\begin{equation}\label{eq:master-tv}
  \bigl\|X^{t}_{\mathrm{id}} - \pi\bigr\|_{\mathrm{TV}}
  \;\ge\; 1 \;-\; \frac{\mathrm{Var}_\pi(\ell)}
                       {\Bigl(\binom{n-2t}{2} - \mathbb{E}_\pi(D)\Bigr)^{2}} .
\end{equation}
In particular, write $\gamma = \dfrac{\theta}{1-\theta} = \dfrac{1}{q-1}$ and
\begin{equation}\label{eq:tvthreshold}
  u \;=\; \frac n2 \;-\; \sqrt{\frac{\gamma\,(n-1)}{2}} \;-\; t
    \;=\; \frac n2 \;-\; \sqrt{\frac{n-1}{2(q-1)}} \;-\; t .
\end{equation}
Then, for $u \ge 1$ and $n$ large,
$$
  \bigl\|X^{\,t}_{\mathrm{id}} - \pi\bigr\|_{\mathrm{TV}} \;\ge\; 1 - O\bigl(u^{-2}\bigr) ,
$$
so the distance tends to $1$ along any sequence of integers $t = t(n)$ with
$u \to \infty$.
\end{theorem}

\begin{proof}
Set $A = \bigl\{\sigma \in S_n : D(\sigma) < \binom{n-2t}{2}\bigr\}$. By
Corollary~\ref{cor:reduce}, $\ell(X^{t}) \le 2tn-2t^{2}-t$ for every realization. That is,
$D(X^{t}) \ge \binom{n-2t}{2}$, so $\mathbb{P}_{\mathrm{id}}(X^{t} \in A) = 0$. On the other
hand, $\sigma \notin A$ forces
$$
  \bigl|D(\sigma) - \mathbb{E}_\pi(D)\bigr|
  \;\ge\; \binom{n-2t}{2} - \mathbb{E}_\pi(D) \;>\; 0 ,
$$
so Chebyshev's inequality gives
$\pi(S_n \setminus A) \le \frac{\mathrm{Var}_\pi(D)}{\bigl(\binom{n-2t}{2}-\mathbb{E}_\pi(D)\bigr)^{2}}$
and hence
$$
  \bigl\|X^{t}_{\mathrm{id}}-\pi\bigr\|_{\mathrm{TV}}
  \;\ge\; \pi(A) - \mathbb{P}_{\mathrm{id}}\bigl(X^{t}\in A\bigr)
  \;=\; \pi(A)
  \;\ge\; 1 - \frac{\mathrm{Var}_\pi(\ell)}
                   {\bigl(\binom{n-2t}{2}-\mathbb{E}_\pi(D)\bigr)^{2}} ,
$$
using $\mathrm{Var}_\pi(D) = \mathrm{Var}_\pi(\ell)$. Replacing Chebyshev's inequality by
Markov's gives 
$$
  \bigl\|X^{t}_{\mathrm{id}}-\pi\bigr\|_{\mathrm{TV}}
  \;\ge\; 1 - \frac{\mathbb{E}_\pi(D)}{\binom{n-2t}{2}}
  \;\ge\; 1 - \frac{n}{(q-1)\dbinom{n-2t}{2}},
$$
valid whenever $n-2t \ge 2$, which is the form quoted in Theorem~\ref{T-tv}; the second
inequality uses $\mathbb{E}_\pi(D) \le \frac{(n-1)\theta}{1-\theta} = \frac{n-1}{q-1}$,
which is \eqref{eq:EDexact} together with $\theta \in (0,1)$.

For the last assertion, put $\sigma = \sqrt{\frac{\gamma(n-1)}{2}}$, so that
$2\sigma^{2} = \gamma(n-1)$, and write $t = \frac n2 - s$ with $s = \sigma+u$. Then
$n-2t \ge 2s$, and since $x \mapsto \binom x2$ is increasing for $x \ge 1$,
$$
  \binom{n-2t}{2} \;\ge\; \frac{2s\,(2s-1)}{2} \;=\; 2s^{2}-s .
$$
By \eqref{eq:EDexact}, $\mathbb{E}_\pi(D) \le \gamma(n-1) = 2\sigma^{2}$, so
$$
  \binom{n-2t}{2} - \mathbb{E}_\pi(D)
  \;\ge\; 2s^{2}-2\sigma^{2}-s
  \;=\; 4u\sigma + 2u^{2} - \sigma - u .
$$
For $u \ge 1$ one has $4u\sigma-\sigma \ge 3u\sigma$ and $2u^{2}-u \ge 0$, so this is at
least $3u\sigma$, which is positive; hence \eqref{eq:master-tv} applies. Since
$\sigma^{2} = \Theta(n)$ and, by Lemma~\ref{lem:statistics},
$\mathrm{Var}_\pi(\ell) = \frac{\gamma(n-1)}{1-\theta}+O(1) = O(n)$, the subtracted term in
\eqref{eq:master-tv} is at most
$\frac{O(n)}{\big(3u\sigma\bigr)^{2}} = O\bigl(u^{-2}\bigr)$.

\end{proof}

Together with the upper bounds of Section~\ref{sec:boundeigen} this brackets the mixing
time from both sides.

\begin{corollary}\label{cor:tvmix}
Fix $q>1$ and let $X$ be $P$, for any $1\le k\le n-1$, or $R$, or $K$.
For every $\delta\in(0,1)$, there exist constants $C_\delta$ and $c_\delta$,
depending only on $\delta$ and $q$, such that, for all sufficiently large $n$,
\begin{equation}\label{eq:bracket}
\frac n2-\sqrt{\frac{n-1}{2(q-1)}}-C_\delta
< t_{\mathrm{mix}}(\delta)
< \frac n2+\frac12\log_q n+c_\delta.
\end{equation}
In particular,

\[
\lim_{\alpha\to\infty}
\lim_{n\to\infty}
d_n\left(\frac{n}{2}-\alpha\sqrt{n}\right)
=1,
 \qquad
\lim_{\alpha\to\infty}
\lim_{n\to\infty}
d_n\left(\frac{n}{2}+\alpha\sqrt{n}\right)
=0.
\]

Thus the three chains exhibit total variation cutoff at $\frac n2$ with a
window of at most order $\sqrt n$.
\end{corollary}

\begin{proof}
For the lower bound in \eqref{eq:bracket}, Theorem~\ref{thm:l1} gives

$$
\|X^t_{\mathrm{id}}-\pi\|_{\mathrm{TV}}\ge 1-O(u^{-2}),
$$

where $u$ is as in \eqref{eq:tvthreshold}. Taking $u\ge C_\delta$ gives the
desired lower bound.

For the upper bound, let $t_c=\frac12(n+\log_q n+c)$. By
Theorem~\ref{thm:upper}, Corollary~\ref{cor:r2r}, or
Theorem~\ref{thm:scan}, according to $X$, choose $c$ so that

$$
\lim_{n\to \infty}
\left\|\frac{X^{ t_c}_{\mathrm{id}}}{\pi}-1\right\|_2
\le 2\delta.
$$

Then \eqref{eq:tvl2} and Theorem~\ref{T-worst} give, for all sufficiently
large $n$,

$$
\max_x\|X_x^{ t_c}-\pi\|_{\mathrm{TV}}\le\delta,
$$

which yields the upper bound in \eqref{eq:bracket}.

Now take $t=\frac{n}{2}-\alpha\sqrt n$. Then

$$
u\ge
\Bigl(\alpha-\frac1{\sqrt{2(q-1)}}\Bigr)\sqrt n+O(1)\longrightarrow\infty
$$

whenever $\alpha>\frac{1}{\sqrt{2(q-1)}}$, so Theorem~\ref{thm:l1} gives
$d_n(t)\to1$.

On the other hand, for every $\alpha>0$,

$$
\frac12\log_q n+c_\delta\le\alpha\sqrt n
$$

for all sufficiently large $n$. Hence \eqref{eq:bracket} gives
$d_n( \frac{n}{2}+\alpha\sqrt n)\le\delta$, and letting $\delta\to0$
gives the upper limit. The two double limits, and hence the cutoff statement, follow.
\end{proof}

\begin{proof}[Proof of Theorem~\ref{T-tv}]
By Theorem~\ref{thm:l1}, for every integer $t$ such that $n-2t\ge 2$,
\[
\left\|X_{\mathrm{id}}^t-\pi\right\|_{\mathrm{TV}}
\ge
1-
\frac{n}
{(q-1)\binom{n-2t}{2}}.
\]

Now take
\[
t=\frac{n}{2}-\alpha\sqrt{n},
\]
with integer parts understood. Then
\[
n-2t=2\alpha\sqrt{n}+O(1),
\]
and therefore
\[
\binom{n-2t}{2}
=
2\alpha^2 n+O(\sqrt{n}).
\]
It follows that, for every $\alpha>0$,
\[
\lim_{n\to\infty}
d_n\left(\frac{n}{2}-\alpha\sqrt{n}\right)
\ge
1-\frac{1}{2(q-1)\alpha^2}.
\]

For the upper bound, Theorem~\ref{thm:upper}, Corollary~\ref{cor:r2r}, and Theorem~\ref{thm:scan} together with 
\[
2\left\|\mu-\pi\right\|_{\mathrm{TV}}
\le
\left\|\frac{\mu}{\pi}-1\right\|_2,
\]
give, for every $\alpha>0$,
\[
d_n\left(\frac{n}{2}+\alpha\sqrt{n}\right)
\longrightarrow 0
\qquad\text{as } n\to\infty.
\]

\end{proof}

\begin{remark}\label{rem:window}
The $\sqrt n$ window comes from using the almost sure bound $D(X^t)\geq \binom{n-2t}{2}$ from Corollary~\ref{cor:reduce}. This bound stops being useful when
$\binom{n-2t}{2}$ is of the same order as
$\mathbb{E}_{\pi(D)}=\Theta(n)$, namely when $n-2t=O(\sqrt n)$. A smaller window
would require finer control of $\mathbb{E}_{\mathrm{id}}(D(X^t))$.
\end{remark}

\begin{remark}\label{rem:identity}
The lower bounds are proved from the identity, which is enough to bound
$t_{\mathrm{mix}}$ from below. The linear term comes from leaving the identity:
by Theorem~\ref{thm:upperavg}, a chain started from a $\pi$--typical permutation
mixes in $\log_q n+c$ steps, since such a permutation already has length
$\binom n2-O(n)$.
\end{remark}



We expect that the $O(\sqrt n)$ in \eqref{eq:bracket} is an artifact of the method, and
that the total variation mixing time in fact sits at the same place as the $\ell^{2}$ one.

\begin{Con}\label{conj:tvcutoff}
Fix $q>1$ and let $X$ be $P$, for any $1\le k\le n-1$, or $R$, or $K$, and put
$t^{\pm}_{c} = \frac12\bigl(n+\log_q n \pm c\bigr)$. Then $X$ has total variation cutoff
at $\frac12\bigl(n+\log_q n\bigr)$ with window $O(1)$:
$$
  \lim_{c\to\infty}\ \lim_{n\to\infty}\
    d_n\bigl( t^{-}_{c}\bigr) \;=\; 1,
  \qquad
  \lim_{c\to\infty}\ \lim_{n\to\infty}\
    d_n\bigl( t^{+}_{c}\bigr) \;=\; 0 .
$$
Equivalently, $t_{\mathrm{mix}}(\delta) = \frac12\bigl(n+\log_q n\bigr)+O(1)$ for every
fixed $\delta\in(0,1)$: nothing is gained by passing from $\ell^{2}$ to
total variation.
\end{Con}

The second limit has already been proved: it is Theorem~\ref{thm:upper} (respectively
Corollary~\ref{cor:r2r}, Theorem~\ref{thm:scan}) together with \eqref{eq:tvl2} and
Theorem~\ref{T-worst}. Only the first is open, and by
Corollary~\ref{cor:tvmix} it is open only in the range
$\frac n2 - O(\sqrt n) \le t \le \frac n2 + \frac12\log_q n$.



\section*{Acknowledgements}

We are deeply grateful to Persi Diaconis for many generous and illuminating conversations, from which this work has benefited greatly. We also thank our advisor, Evita Nestoridi, for her invaluable guidance throughout
this project. The first author thanks Colin Defant for useful discussions.

\paragraph{Statement on the Use of Artificial Intelligence.}
All of the mathematics in this paper is due to the authors: no proof was obtained with the assistance of artificial intelligence, and the authors are solely
responsible for the correctness of every statement and proof. The authors used large language models to assist with proofreading, improve the readability of the exposition,
check the internal consistency of notation, labels and cross--references, and run numerical simulations.

\printbibliography

\end{document}